\documentclass[11pt, reqno]{amsart}

\usepackage{amsmath,amssymb,amsthm}
\usepackage{shadethm}
\usepackage{dsfont}
\usepackage{mathrsfs}
\usepackage[utf8]{inputenc}
\usepackage[T1]{fontenc}
\usepackage{mathptmx}
\usepackage{appendix}
\usepackage{hyperref}
\usepackage{bm}
\usepackage{verbatim}
\usepackage{enumitem}
\usepackage{bbm}
\usepackage{color}
\DeclareMathAlphabet{\mathcal}{OMS}{cmsy}{m}{n}
\usepackage{color}
\definecolor{shadethmcolor}{cmyk}{0,0,0,0.05}
\definecolor{shaderulecolor}{cmyk}{0,0,0,0}

\theoremstyle{definition}
\newtheorem{definition}{Definition}[section]

\newshadetheorem{theorem}[definition]{Theorem}
\newshadetheorem{proposition}[definition]{Proposition}
\newtheorem{lemma}[definition]{Lemma}
\newtheorem{construction}[definition]{Construction}
\newtheorem{remark}[definition]{Remark}
\newtheorem{remark+def}[definition]{Remark \& Definition}
\newtheorem{corollary}[definition]{Corollary}
\newtheorem{example}[definition]{Example}

\DeclareMathOperator{\capa}{cap}
\DeclareMathOperator{\var}{var}

\newcommand{\black}{\color[rgb]{0,0,0}}

\renewcommand{\b}[1]{\operatorname{\mathbb{#1}}}

\newcommand{\overbar}[1]{\mkern 1.5mu\overline{\mkern-1.5mu#1\mkern-1.5mu}\mkern 1.5mu}

\newcommand{\sumtwo}[2]{\sum_{\substack{#1 \\ #2}}} 

\newcommand{\cadlag}{c\`adl\`ag}

\newcommand{\sfrac}[2]{\mbox{$\frac{#1}{#2}$}}

\newcommand{\cC}{\mathcal{C}}
\newcommand{\cD}{\mathcal{D}}

\newcommand{\cF}{\mathcal{F}}
\newcommand{\cG}{\mathcal{G}}

\newcommand{\cN}{\mathcal{N}}
\newcommand{\cO}{\mathcal{O}}
\newcommand{\cP}{\mathcal{P}}
\newcommand{\cQ}{\mathcal{Q}}

\newcommand{\cX}{\mathcal{X}}

\newcommand{\E}{\mathbb{E}}
\newcommand{\IH}{\mathbb{H}}
\newcommand{\IG}{\mathbb{G}}
\newcommand{\II}{\mathcal{V}}
\renewcommand{\IJ}{\mathcal{E}}

\newcommand{\enter}{e}

\renewcommand{\P}{\mathbb P}

\newcommand{\N}{\mathbb N}

\newcommand{\1}{\operatorname{\mathbbm{1}}}

\renewcommand{\b}[1]{\operatorname{\mathbb{#1}}}

\renewcommand{\c}[1]{\operatorname{\mathcal{#1}}}

\newcommand{\dd}{\mathrm{d}}

\begin{document}
\title[Traces left by random walks]{Traces left by random walks on large random graphs: local limits}

\author[Steffen Dereich]{Steffen Dereich}

\maketitle

\vspace{-0.3cm}

\begin{quote}
{\small {\bf Abstract:} }
In this article, we develop a theory for understanding the traces left by a random walk in the vicinity of a randomly chosen reference vertex. The analysis is related to interlacements but goes beyond previous research by showing weak limit theorems for the vicinity around the reference vertex together with the path segments of the random walk in this vicinity.
Roughly speaking,  our limit theorem requires appropriate mixing properties for the random walk together with a stronger  variant of Benjamini-Schramm convergence for the underlying graph model. If these assumptions are satisfied, the limiting object can be explicitly given in terms of the  Benjamini-Schramm limit of the graph model. 
\end{quote}

\vspace{0.5cm}

{\footnotesize
\vspace{0.1cm}
\noindent\emph{MSc Classification:}  Primary 05C81 
Secondary 05C82, 
60C05, 
90B15. 

}

\vspace{0.5cm}

\maketitle
\tableofcontents
\newpage

\section{Introduction}

In this article, we analyse reversible  random walks on large random or deterministic finite, rooted graphs. The aim is to understand the traces left by the  random walk  in the neighbourhood of the root, where typically the root vertex is a uniformly chosen vertex of the underlying graph. 
In our analysis,  we assume a strengthened version of  Benjamini-Schramm convergence \cite{Benjamini2000}, see also~\cite{AldousSteele}, and we impose an assumption on the mixing times of the random walks.

Local convergence techniques are the main tool for understanding the asymptotic behaviour  of large random graphs. For instance, the relative size of the largest cluster typically converges to the probability of the limitting graph (appearing as local limit) being infinite, see
\cite{ErRe61} for the Erd\H{o}s-R\'{e}nyi model and  \cite{MolRee98} and \cite{JansonLuczak2007} for results concerning random graphs with fixed degree sequence  (also known as configuration model). Similar results hold when applying percolation on the graphs.
 The theory is nowadays well established and it is the topic of several monographs, see~\cite{AlSp00, JLR00, Dur07, Hofstad2016i}.

More recently,  \cite{Hofstad21} showed that the relationship between the relative size of the largest component and the local limit becomes more transparent when replacing local convergence by a stronger type of convergence that is coined \emph{local convergence in probability} in the article. We face a similar problem in this article since the quantities that we need to control in the graph are not continuous in the topology of local convergence. A remedy for this will be a stronger type of  convergence. 


Since we consider a random walk on the largest component the respective percolation problem is \emph{vacant set percolation}, where one removes all sites seen by a random walk and then asks whether a large component remains. As observed by Sznitman in \cite{Szn10,SidSzn10} so called interlacements are central towards the understanding of such problems. 

When considering vacant set percolation on the configuration model one obtains analogies of the results known for classical percolation, where the relative size of the largest component for the vacant set equals the probability that in the limiting random rooted graph the root is in an infinite component after removal of all the sites seen by an interlacement, 
see \cite{CernyTeixeiraWindisch2010, CernyTeixeira2011,  CerHay20} for results concerning the configuration model. We stress that from a technical point of view the analysis of vacant set percolation often make strong use of  particularities of the configuration model as established in \cite{JansonLuczak2007}.


Our primary focus does not lie  on percolation (although we will explain how one can extend the approach taken in \cite{Hofstad21} to vacant set percolation). Instead we record information on the behaviour of the random walk around a possibly randomly picked root. We take the role of an observer placed in the root that  tracks the movement of the random walk once it comes into a certain distance to the observer for a certain amount of time. We prove limit theorems in total variation distance for the traces  seen by the observer. From a statistical point of view the distribution of the limiting object can be used as a prior distribution for filtering applications. An easy example for a statistic that then will converge in distribution is a function of the time spent in the vertices around the root. Significantly more elaborate statistics are included.  

From a mathematical point of view the problem is that incoming new paths may come from various directions and the probabilities governing this process are \emph{not} local quantities. To cope with this we demand to have control about the goodness of the local approximation (of the graph) which has to increase sufficiently fast as the size of the graph tends to infinity.  We formalise the meaning of the latter and we will relate this to the mixing times of the random walks.  Roughly speaking, one needs to be able to couple an exploration of the graph along the random walk started in the root for at least the mixing time. The analysis relies on new  estimates for the total variation distance of certain entrance locations and times  to an idealised decoupled random variable. The estimates are in spirit similar to the ones derived in \cite{Aldous1992}, but unknown to the best of our knowledge.



%
%

\subsection{The setup}

We consider random walk in random environment, where the environment is given in terms of a connected, undirected graph $G$ with vertex- and edge-weights.
For technical reasons we also allow multigraphs.
To be formally precise we let $\II$ and~$\IJ$ be disjoint, countable, totally ordered\footnote{The ordering allows us to later define particular operations on the graphs. The particular choice of the ordering has no impact on the results of the article.}  sets and denote by $\partial$ a cemetery state that is neither in $\II$ nor $\IJ$. We endow $\II$ and $\IJ$ with the discrete topology and conceive $\II\cup \{\partial\}$ and $\IJ\cup\{\partial\}$ as one-point compactifications of the latter sets. We will use $\II$ and $\IJ$ for the labels of the vertices and edges, respectively.  Formally  a weighted \emph{multigraph} (briefly \emph{$\IG$-graph}) is a tuple 
$G=(V,E,\psi,\alpha,\beta)$ consisting of
\begin{enumerate} \item[(i)] a non-empty set  $V\subset \II$, the set of \emph{vertices}, and $E\subset \IJ$,  the set of \emph{edges},
\item[(ii)]  a mapping $\psi$ that maps every $e\in E$ to an unordered $V$-valued pair $$\psi(e)=\langle  e',e''\rangle,$$ such that every vertex $v$ appears in at most finitely many pairs $\psi(e)$ ($e\in E$), 
\item[(iii)]  a mapping $\alpha:V\to(0,\infty)$, the \emph{vertex-weights}, and
\item[(iv)] a mapping $\beta:E\to(0,\infty)$, the \emph{edge-weights} $\beta$.
\end{enumerate} 
We denote by $\IG$ the set of all  such graphs. Formally, a \emph{graph isomorphism} taking a graph $G$ to a graph $\bar G$ is a mapping  $\varphi:V\cup E\to \bar V\cup \bar E$ such that $\varphi|_V:V\to \bar V$ and $\varphi|_{E}:E\to \bar E$ are one-to-one and for all $v\in V$ and $e\in E$
$$
\bar\alpha_{\varphi(v)}= \alpha_v, \, \bar \beta_{\varphi(e)}=\beta_e \text{ \ and \ }\bar \psi(\varphi(e))=\langle \varphi(e'), \varphi(e'')\rangle.
$$
A rooted graph is a graph $G\in\IG$ together with a designated vertex $o\in V$ and we briefly write for the graph $G$ with root $o$, $G^o$. We denote by $\IG^\mathrm{root}$ the set of rooted multigraphs and, additionally,  write $\IG^\mathrm{finite}$ ($\IG^\mathrm{finite,\, root}$) for the finite (rooted) $\IG$-graphs.

 A \emph{graph isomorphism} $\varphi$ between two rooted graphs $G^o$ and $\bar G^{\bar o}$ is a graph isomorphism that takes $G$ to $\bar G$ and additionally satisfies $\bar o=\varphi(o)$. We call two rooted or unrooted graphs \emph{equivalent} if there is a graph isomorphism taking one to the other. The concept agrees with the one used in \cite{Hof21} up to the additional weights and the fact that our graphs have explicit labels (from $\mathcal V$). Both will be needed to formally introduce the respective random walks. 

A \emph{generalised multigraph} is a tuple $G=(V,E,\psi,\alpha,\beta)$ where condition (ii) above is replaced by
\begin{enumerate}
\item[(ii')] a mapping $\Psi$ that maps every $e\in E$ to an unordered $V\cup\{\partial\}$-valued pair
$$
\psi(e)=\langle e',e''\rangle \not=\langle \partial,\partial\rangle,
$$
such that every vertex $v$ appears in at most finitely many pairs $\psi(e)$ ($e\in E$)
\end{enumerate}
and we write $\bar \IG$ for the respective set of graphs. In generalised graphs 
In analogy to above, we write $\bar\IG^\mathrm{finite,\, root}$ for the generalised graphs with designated root.

We let $\overline W_+$ denote the space of all paths $w:[0,\infty)\to \II\cup\{\partial\}$ for which there exists $t^\partial\in(0,\infty]$ such that  for all $t\in[0,\infty)$
$$
t<t^\partial   \ \Leftrightarrow \  w_t\in \II.
$$
We associate a graph $G\in\IG$ with the continuous time Markov process family $(P^x_G:x\in V)$ (with all paths lying in $W_+$) with generator 
 $L=L^G=(L^G_{x,y})_{x,y\in V}$  given by
\begin{align}\label{def:gen}
\begin{cases} L_{x,y}^G=  \sum_{e: \langle e',e''\rangle =\langle x,y\rangle}\beta_{e}/ \alpha _x, & \text{ for } x\neq y, \\
L^G_{x,x}=-\sum_{y\not =x} L^G_{x,y}=: -L^G_x, &  \text{ for } x\in V .\end{cases}
\end{align}
The respective Markov process will be called \emph{random walk associated with $G$}.
Note that for every distribution $\mu$, a $P^\mu_G$-Markov process can be generated iteratively as follows: let  $X_0$ be a $\mu$-distributed random variable and then iteratively repeat the following
\begin{itemize}
\item[(i)] stay in the current state, say  $v\in V$, for  an $\mathrm{Exp}(L_{v}^G)$-distributed amount of time and 
\item[(ii)] then choose an edge $e\in \{\langle e',e''\rangle: v\in \{e',e''\}\} $ with probability
proportionally to~$\beta_e$ and change the current state along the edge $e$ (meaning that the new state is just the endpoint of $e$ that is not equal to  $v$).
\end{itemize}
We denote by $S_1,\dots$ the time points at which transitions according to (ii) are carried out and set  $X_t=\partial$ for all $t\ge \sup_{n\in\N} S_n$.
We denote by $(X_t)_{t\ge0}$ the canonical process on the Markov family $(P^v:v\in V)$. In view of the above construction we mention that, in the case that the $L_G$-Markov process has multiple edges or loops, the edges along which transitions occur cannot be recovered by the states of the Markov process $(X_t)_{t\ge0}$ itself and we suppose that in the background there is always a construction as above and we can speak about the first, second, .. transition and assign each transition to a unique edge of the multigraph as above.

We note that by construction the Markov-process is $\alpha$-reversible. If the graph is connected and if $\alpha(V)=\sum_{v\in V} \alpha_v$ is finite, then the Markov process has as unique invariant distribution 
$$
\pi=\pi^G=\sfrac 1{\alpha(V)}\alpha.
$$
In this case, we denote for $\tau\ge0$
$$
d_G(\tau)= \sup_{v\in V}\, d_\mathrm{TV}(P_G^v(X_t\in \,\cdot\,),\pi^G).
$$
\black

\subsection{The visiting counting measure}\label{sec_visit}


Let $G\in\IG$. Now we introduce the central object that records  information about the visits to a certain region of the graph. Its definition relies on three parameters: the \emph{entrance region} $B\subset V$, the \emph{observation time} $\tau>0$ and the \emph{scaling factor} $a>0$. 
For a given choice of parameters   a trajectory $X=(X_t)_{t\geq 0}\in \overbar W_+$ (that later will be a random walk on the random environment $G$) is mapped to a counting measure on $\b S:=[0,\infty)\times W_+$, where $W_+$ denotes the space of \emph{transient paths}, i.e., 
\begin{align*}
W_+&=\bigl\{(w_t)_{t\geq 0}\in\overbar W_+ : \forall x \in \II \text{ \ the set  } \{t\in[0,\infty): w_t=x\} \text{ is bounded}\bigr\}\\
&=\bigl\{(w_t)_{t\geq 0}\in\overbar W_+ : \lim_{t\to\infty} w_t=\partial\bigr\}.
\end{align*}
  We set 
$$
T_1= \inf\{ t\geq 0 : X_t\in B\}
$$
and $$T_k=  \inf\{ t> T_{k-1}+\tau : X_t\in B\}, \text{ \ for \ }k=2,3,\dots$$
and call the sequence $(T_k:k\in\ N)$ the sequence of \emph{$(B,\tau)$-entrance times} of $X$. Further we denote for every  $k\in \N$ with $T_k<\infty$ by $(Y_k(t))_{t\geq 0}$  the $k$th \emph{$(B,\tau)$-visit} of $X$  given by
$$
Y_{k}(t)=\begin{cases} X_{T_k+t}, &\text{ if } t <  \tau , \\
\partial, & \text{ else.}\end{cases}
$$
For completeness we set for $k\in\N$ with $T_k=\infty$, $Y_k\equiv \partial$.
Finally, we define the \emph{$(B,\tau,a)$-visiting measure} to be the counting measure on $\b S$ given by
$$
\Xi=\sumtwo{k\in\N:}{T_k<\infty} \delta_{T_k/a, Y_k}.
$$
Roughly speaking, a new (macroscopic) visit to $B$ is initiated when the path hits $B$ again after time $\tau$ has passed since the last visit and the $k$-th visit is just the time changed trajectory of $X$ in the time interval $[T_k,T_k+\tau]$.
Note that the previous definitions make  perfect sense when the set $B$ and the graph $G$ are random.
%
%


\subsection{The main result}


We will use the concept of exploration convergence for a sequence of random $\IG^\mathrm{root}$-graphs  $(G_n^{o_n}:n\in\N)$ as introduced in Section~\ref{sec:expl_conv} below. 
Intuitively, a sequence of random $\IG^\mathrm{root}$-graphs $(G_n^{o_n}:n\in\N)$ converges $(\ell_n)$-locally to a random rooted standard graph $G^o$ (for an $\N$-valued sequence $(\ell_n)$) if exploration processes that collect information about the two 
 graphs $G_n^{o_n}$ and $G^o$ step by step can be coupled in such a way that they find the same structure modulo graph isomorphism with high probability as long as they are run for $\ell_n$ steps.

From now on for each $n\in\N$, $G_n^{o_n}$ denotes a random finite, connected  $\IG^\mathrm{root}$-graph  and  $X=X^{(n)}=(X^{(n)}_t)_{t\geq 0}$ denotes a \cadlag\ process on $\b I$ such that under $\P$ given  $G_n^{o_n}$ the process $X^{(n)}$ is an $L_{G_n}$-Markov process started in the unique $L_{G_n}$-invariant distribution $\pi_n:=\pi^{G_n}$. Moreover, $G^o$ will denote a random connected $\IG^\mathrm{root}$-graph. We denote for $R\in\N_0$, by $V^o[R]$ 
the vertices of $G$ 
that have graph distance (hop-count distance) less or equal to $R$ to the root $o$. Moreover,  $\partial V^o[R]$ denotes the set of vertices of $G$ with graph distance  $R$ to the root vertex. We use analogous notation for the graphs $G_n$ and write $V_n^{o_n}[R]$ and $\partial V_n^{o_n}[R]$ for the respective sets.

%
%
\medskip
\begin{theorem}[Coupling visiting measures] \label{prop:2746} 
Let $(\ell_n')_{n\in\N}$ and $(\ell_n)_{n\in\N}$ be $\N$-valued and $(\tau_n)_{n\in\N}$ and $(a_n)_{n\in\N}$ be $(0,\infty)$-valued with all three sequences tending to infinity and suppose that the random finite connected  $\IG^\mathrm{root}$-graphs  $G_1^{o_1},\dots$  converge $(\ell_n)$-locally to the connected $\IG^\mathrm{root}$-graph $G^o$. Moreover, suppose that $\ell_n'=o(\ell_n)$, $\tau_n=o(\alpha_n(V_n))$, 
$$
\lim_{n\to\infty} d_{G_n}(\tau_n)=0, \ \ \lim_{n\to\infty} P_{G_n}^{o_n}\bigl(\# \text{ range}\bigl((X_t)_{[0,2\tau_n]}\bigr) >\ell_n'\bigr)   =0\text{ \ and }
$$
$$
\lim_{n\to\infty}\frac{\alpha_n(V_n)}{a_n}=\mathfrak a ,\text{ in probability.}
$$
Let $R,R'\in\N$ and $\sigma>0$. For $n\in\N$, let $\Xi_n$ be the $(V_n^{o_n}[R],\tau_n,a_n)$-visiting measure and let $\Gamma$ be the Cox-process that is conditionally on $G^o$ a Poisson point process with intensity
\begin{align}\label{eq94624}
\sfrac 1{\mathfrak a} \mathrm{Leb}|_{[0,\infty)} \otimes P^{e_{V^{o}[R],G}}_G
\end{align}
and let $\Gamma_n$ be the point process obtained from $\Gamma$ by killing all paths at time $\tau_n$.
 One can couple $(G^{o_n}_n,\Xi_n)$ and $(G^o,\Gamma)$ such that, the following holds, with high probability:
\begin{itemize}
\item There are subsets $\bar{\mathcal V}_{n}\supset V_n^{o_n}[R']$ and ${\mathcal V}_{n}\supset V^o[R']$ of $V_n$ and $V$, respectively, and a graph isomorphism taking the $\bar {\mathcal V}_n$-disclosed graph of $G_n^{o_n}$ to the $\mathcal V_n$-disclosed graph of $G^o$.
\item The trajectories appearing in $\Xi_n|_{[0,\sigma]\times W_+}$ do not leave $\bar {\mathcal V}_n$ and the point process $\Gamma_n|_{[0,\sigma]\times W_+}$ is obtained from  $\Xi_n|_{[0,\sigma]\times W_+}$ by applying the isomorphism  $\varphi_n$ on the states of the trajectories. 
\end{itemize}
\end{theorem}

\begin{remark}[Interlacement representation] \label{rem:235462} Following \cite{Szn10} the measure $P_G^{e_{V^o}[R],G}$ appearing in the theorem admits a description in terms of interlacements. Let $\bar W$ be the space of all \cadlag\  paths $w:\mathbb R\to \II\cup\{\partial\}$ such that there exist $-\infty\le t_\partial <t^\partial\le \infty$ so that for all $t\in\b R$
$$
t_\partial \le t<t^\partial \ \Leftrightarrow \  w_t\in \II. 
$$
We let $W$ denote the subspace of $\overline W$ consisting of transient paths, i.e.,
$$
W=\{(w_t)_{t\in\b R} \in \overline W: \forall v\in \II \text{ the set } \{t\in\b R:w_t=v\}\text{ is bounded}\}
$$
We call two paths $(w_t)$ and $(\bar w_t)$ of $W$ equivalent ($\sim$), if there exists $\tau\in\b R$ such that $w_t=\bar w_{\tau+t}$ for all $t\in\b R$. Note that $\sim$ is an equivalence relation and we denote by $W^*=W/\!\sim$ the set of equivalence classes. Following \cite{Szn10} we equip $W^*$ with the finest  $\sigma$-field such that the embedding $$W\ni w \mapsto [w]\in W^*$$
is measurable\footnote{Here $[w]$ denotes the $\sim$-equivalence class of $w$.}.
Then for every  connected fixed graph $G^o\in \IG^\mathrm{root}$ for which $P_G$ is transient we have that there exists a $\sigma$-finite measure $\nu_G$ (an \emph{interlacement}) on $W^*$ such that for every $R\in\N$ 
\begin{align}\label{eq8472}
P_G^{e_V^o[R],G} = \nu_G|_{W^*_{V^o[R]}} \circ \pi_{V^o[R]}^{-1},
\end{align}
where $W^*_{V^o[R]}$ denotes the set of transient paths modulo $\sim$ that enter $V^o[R]$ and $$\pi_{V^o[R]}:W^*_{V^o[R]}\to W_+, \, [(w_t)_{t\in\b R} ]\mapsto (w_{T_{V^o[R]}(w)+t})_{t\ge0}$$ 
sends an equivalence class $[(w_t)]\in W^*_{V^o[R]}$ to the trajectory that is following the trajectory $(w_t)$ from its first entrance into $V^o[R]$ onwards. Based on~(\ref{eq8472}) we get an alternative representation  for the intensity measure (\ref{eq94624}) based on interlacements.
\end{remark}

\black

A particular family of random graphs for which the assumptions of the main theorem can be verified are sparse Erd\H{o}s-R\'enyi random graphs and random graphs  with fixed degree sequence.

\begin{example}[Random graphs with fixed degree sequence]\label{exa2}
For every $n\in\N$ we denote by $\mathbf{d}^{(n)} =(d_k^{(n)})_{k=1,\dots,n}$ an $\N_0$-valued sequence so that $N^{(n)}:=\sum_{k=1,\dots,n} d_k^{(n)}$ is even and strictly positive. The random graph model with fixed degree sequence $(\mathbf d^{(n)})_{n\in\N}$ is constituted by a sequence of random multigraphs that is formed by pairing the  half-edges in 
$$
\IH_n=\bigcup_{j\in\{1,\dots,n\}}  \{j\} \times \{1,\dots, d_j^{(n)}\}
$$ 
and forming the random multigraph  $G_n$ by taking as vertex set $\{1,\dots,n\}$ and establishing for each pair $\langle (j_1,r_1),(j_2,r_2)\rangle$ of the latter random pairing an edge connecting $j_1$ and $j_2$. The edge-weights are chosen to be one and the vertex-weights are chosen as the respective degrees.

We denote by $\cC_n^\mathrm{max}$ the largest component of $G_n$ and let $\bar G_{n}^{\bar o_{n}}$ denote the graph $G_{n}$ restricted to $\cC_n^\mathrm{max}$ equipped with a uniformly chosen root.
As parameters of the generator of the underlying random walk we choose $\alpha$ to be the degree and $\beta$ to be constantly one.

\begin{enumerate}
\item \emph{Mixing times.} Provided that for large $n\in\N$,  
\begin{align}\label{eq78234} 3\le \min_{j=1,\dots,n} d_j^{(n)}\le \max_{j=1,\dots,n} d_j^{(n)}\le n^{0.02}\end{align}
 one has that there exists $\alpha>0$ such that, with high probability, $G_n$ is an $\alpha$-expander \cite[Lemma 5.3]{BKW14}. Moreover, with high probability, $G_n$ is connected. This entails that provided that for a finite constant $C$, with high probability, $\sum_{j=1}^n d_j^{(n)}\le C n$, one has w.h.p.\ that the minimal mass of the stationary distribution tends to zero of order $\frac 1n$ and the mixing time is of order $\mathcal O(\log n)$. 
\item \emph{Local convergence.} We denote by $D_n$ an $\N_0$-valued random variable with distribution
$$
\frac 1n\sum_{k=1}^n \delta_{d_k^{(n)}}
$$ 
and let $D_n^*$ denote a $\N$-valued random variable with size biased  distribution
$$
\frac 1{N^{(n)}}\sum_{k=1}^n d_k^{(n)} \delta_{d_k^{(n)}}.
$$ 
As is shown in the appendix (see Theorem~\ref{theo:12} and Remark~\ref{rem:246}.(a)) we have the following: Suppose  that for an $\N$-valued sequence $(\ell_n)_{n\in\N}$ and for $\N_{0}$-valued random variables $D$ and $D^*$, one has 
$$
\ell_n=o\bigl(\sqrt n \, \E[D_n^*\wedge \sqrt n]^{-1}\bigr)
$$
 and 
 \begin{align}\label{eq:73485}\lim_{n\to\infty}d_\mathrm{TV}(\P_{D_n},\P_D)+\ell_n \, d_\mathrm{TV}(\P_{D_n^*},\P_{D^*})=0.
\end{align}
Then the respective sequence $(G_{n}^{o_{n}})_{n\in\N}$ of uniformly rooted random graphs with degree sequence $(\mathbf d^{(n)})_{n\in\N}$ converges $(\ell_{n})$-locally to a $\mathrm{GWP}(D,D^{*})$-tree as introduced in the appendix. 

Suppose now that the $\mathrm{GWP}(D,D^{*})$-tree is infinite with strictly positive probability. Then by   Corollary~\ref{cor:823757} the rooted graphs $(\bar G^{\bar o_{n}})$ converge $(\ell_{n})$-locally to the $\mathrm{GWP}(D,D^{*})$ process conditioned to be infinite.
\item \emph{Application of the main theorem.} In the situation where we can apply (a) and (b) above, we choose $a_{n}=n$ and, moreover, $(\tau_{n})$, $(\ell'_{n})$ and $(\ell_{n})$ such that
$$
\lim_{n\to\infty}\frac{\log n}{\tau_{n}}=0, \ \lim_{n\to\infty}\frac{\tau_{n}}{\ell_{n'}}=0, \lim_{n\to\infty}\frac{\ell_{n'}}{\ell_{n}}=0 \text{ and }  \lim_{n\to\infty}\frac{\ell_{n}}{\sqrt n \, \E[D_n^*\wedge \sqrt n]^{-1}}.
$$
This is easily established since assumption (\ref{eq78234}) in  (a) entails that $$\sqrt n\, \E[D_{n^{*}}\wedge \sqrt n]^{-1}\ge n^{0.48}.$$
The number of jumps of the random walk in the time interval $[0,2\tau_{n}]$ is by choice of the walk Poisson-distributed with parameter $2\tau_{n}$ so that obviously the range constraint is satisfied.  
We assume that~(\ref{eq:73485}) is satisfied for random variables $D$ and  $D^{*}$. 

Recall that $G_{n}$ is with high probability connected and
$$
\alpha_{n} (V_{n})=\sum_{k=1}^{n} d_{k}^{(n)}=n\, \E[D_{n}].
$$
Next, we show that $\lim_{n\to\infty}\E[D_{n}]=\E[D]$. Note that for all $k\in\N$,
$\P(D_{n}^{*}=k)=\E[D_{n}]^{-1}k \,\P(D_n=k)$. Pick an index $\ell\in\N$ with $\P(D^{*}=\ell)>0$ and observe that
$$
\lim_{n\to\infty}\E[D_{n}]^{-1}\ell \,\P(D_n=\ell)= \P(D^{*}=\ell)>0
$$
 and 
 $$ \limsup_{n\to\infty} \E[D_{n}]\le \ell \,\P(D=\ell) \,\P(D^{*}=\ell)^{-1}<\infty.
$$
Since
 $\lim_{n\to\infty}d_{\mathrm{TV}}(\P_{D_{n}^{*}},\P_{D^{*}})=0$ we get that
\begin{align*}
\lim_{n\to\infty}\E[D_{n}]^{-1} \sum_{k=1}^{\infty} \bigl| k \,\P(D_n=k)-\E[D_{n}]\,\P(D^{*}=k)\bigr|=0.
\end{align*}
Since $(\E[D_{n}])$ is bounded it follows that
$$
\Bigl| \sum_{k=1}^{\infty}  k \,\P(D_n=k)-\E[D_{n}] \Bigr| \le \sum_{k=1}^{\infty} \bigl| k \,\P(D_n=k)-\E[D_{n}]\,\P(D^{*}=k)\bigr|\to0
$$
so that $\E[D_{n}]\to \E[D]$ with a finite limit. Thus we need to choose $\mathfrak a=\E[D]$.

Consequently, the visiting measure converges to the respective limit process on the $\mathrm{GWP}(D,D^{*})$ tree.
\end{enumerate}
\end{example}

\begin{example}[Erd\H{o}s-R\'enyi random graphs]\label{exa1}We let $(m(n))_{n\in\N}$ be a $\N$-valued sequence such that
$$
\lambda=\lim_{n\to\infty} \frac{m(n)}{n}
$$
exists and lies in he interval $(1,\infty)$. Denote for every $n\in\N$, by $G_n$ a uniformly chosen graph under all (simple) graphs with vertex set $\{1,\dots,n\}$ with $m(n)$ edges\footnote{If $m(n)>{n\choose2}$, then there are no such graphs. In that case we pick the complete graph.}. The edge weights are chosen to be one and the vertex weights are equal to the respective degree. We denote by $\cC_n^\mathrm{max}$ the largest component of $G_n$. As shown in \cite{BKW14}, one has, with high probability, that the mixing time is of order $\Theta\bigl((\log n)^2\bigr)$ for the random  walk on the maximal cluster $\cC_n^\mathrm{max}$ of $G$. Let $(\ell_{n})_{n\in\N}$ be an $\N$-valued sequence with $\ell_n=o(\sqrt n)$.
Similarly, as in the appendix one can show that the uniformly rooted graph $G_{n}^{o_{n}}$ converges $(\ell_n)$-locally to a $\mathrm{GWP}(D,D+1)$-tree, where $D$ denotes a Poisson-distributed random variable with parameter $\lambda$. Then again one can use Corollary~\ref{cor:823757} to deduce the local limit for the graph restricted to the giant component. One still needs to verify the convergence towards an appropriate $\mathfrak a$ in the respective model.
\end{example}

\subsection{
Exploration convergence}\label{sec:expl_conv}

{ 
Classical Benjamini-Schramm convergence is a form of  weak convergence for random rooted graphs. More explicitly, a sequence of random rooted graphs $G_n^{o_n}$ converges to a random rooted graph $G^o$ in the sense of Benjamini-Schramm~\cite{Benjamini2000} if  for every radius $R\in\N$ the rooted graphs can be coupled such that (after relabeling) the $R$-neighbourhoods of $G_n^{o_n}$ agree with $G^o$ in the limit with probability one. 
In this article we work with a different notion of local convergence, one which incorporates a control of the quality of the latter approximation depending on the index $n$. For this we will use the concept of an exploration process.

 We start with the formal definition of an \emph{exploration rule}.
Recall that $\bar {\b G}^{\mathrm{finite,\, root}}$ is the set of finite, rooted, generalised graphs. 
An \emph{exploration} is a tuple $$\mathfrak E=(\mathcal G,(e_1,\dots,e_n))\in \bar {\b G}^{\mathrm{finite,\, root}}\times \IJ^*$$
so that for each $i=1,\dots,n$ the edge $e_i$ is in $\mathcal G$ and the edge $e_i$ has both of its endpoints in the set of vertices in $\mathcal G$. 
 
 An \emph{exploration rule} $K$ is a probability kernel that assigns an exploration $(\mathcal G,(e_1,\dots,e_n))$ to  a distribution on the edges of $\mathcal G$ that is invariant under isomorphisms. More explicitly, for an exploration $\mathfrak E=(\mathcal G,(e_1,\dots,e_n))$ and a graph isomorphism $\phi$ taking the exploration $\mathfrak E$ to $\mathfrak E'=\phi(\mathfrak E)$ (meaning that $\mathfrak E'=(\phi(\mathcal G),(\phi(e_1),\dots,\phi(e_n)))$) we have
 $$
 K(\mathfrak E,\varphi\in \,\cdot\, )= K(\mathfrak E',\,\cdot\,).
 $$

For a graph $G$ and a finite subset $V_0\subset V$ of vertices we build  the \emph{$V_0$-disclosed subgraph} $G_{|V_0}$ of $G$ as follows: it is the generalised weighted finite graph that
\begin{itemize}
\item has vertices $V_0$ and edges 
$\{e\in E(G): \{e',e''\}\cap V_0\not=\emptyset\}$\\ (all edges with one of the endpoints lying in $V_0$)
\item where  every edge connecting to  a vertex outside of $V_0$ is redirected to $\partial$  and
\item all vertex and edge weights are as in $G$.
\end{itemize}

An \emph{exploration process} is now a stochastic process that is defined for an arbitrary  rooted graph $G^o$ and an exploration rule $K$ as follows: $\mathfrak E_{0}=(\mathcal G_0^o, \emptyset)$ where $\mathcal G^o_0$ is the $V(0)=\{o\}$-disclosed subgraph of $G$ with root $o$ and for  every $k=1,\dots$ we iteratively
\begin{itemize}
\item choose an $K(\mathfrak E_{k-1}; \,\cdot \,)$-distributed edge $e_k$,
\item form $V(k)$ by adding to $V(k-1)$ the endpoints of $e_k$ in the graph $G$ (if $e_k=\partial$ no vertices are added) and
\item let $\mathfrak E_{k}=(\mathcal G_k^o,(e_1,\dots,e_k))$ where $\mathcal G_k^o$ denotes the $V(k)$-disclosed subgraph of~$G$ with root $o$.
\end{itemize}
Note that during the exploration the disclosed subgraph changes whenever the exploration rule picks an edge that leaves the set of previously disclosed vertices.
 In the case that $G^o$ is itself random we run the exploration with the respective conditional distributions.

}

\begin{definition}
We call an exploration rule $K$ \textit{proper} if for every $\c E\in \b G^\mathrm{expl}$ with at least one edge connecting to $\partial$ the distribution   $K(\c E, (e_1,\dots, e_k);\,\cdot \,)$ puts all mass on edges connected to $\partial$. 
\end{definition}

\begin{example}[Breadth-first exploration]  
Choosing $K(\c E, (e_1,\dots, e_k); \,\cdot\,)$ as the uniform distribution on all edges connecting to $\partial$ with minimal distance to the root, leads to the \textit{breadth-first} exploration process of a component of the graph.   
\end{example}
 
\begin{example}[Exploration governed by the Markov procecss] \label{exa_MP} Let $G^o\in \IG^\mathrm{root}$.  We denote by  $(X_t)_{t\in[0,\infty)}$ a $P^o_G$-distributed process and let for $k\in\N_0$
$$
S_k=\inf \bigl\{t\ge 0: \# \mathrm{range} \bigl((X_s)_{s\in[0,t]}\bigr)> k\bigr\}.
$$
Moreover, let for every $k\in\N_0$, $e_k$ denote the random edge that is traversed by the Markov process at time $S_k$ if the latter time is finite and $\partial$ otherwise. Note that one can read of the state of the Markov process at time $S_k$ from $\mathfrak E_{k-1}$ and the edge passed by $X$ at that time. Given $(X_t)_{t\in [0,S_k]}$, the conditional distribution of $((X_{t-S_k})_{t\in[0,S_{k+1}-S_k)}, e_{k+1})$ is just the distribution of the Markov process started in $X_{S_k}$ that is killed immediately when leaving the set of disclosed vertices in $\mathfrak E_k$ with $e_{k+1}$ being the edge traversed by the Markov process when exiting. Obviously, the conditional distribution is $\mathfrak E_k$-measurable and invariant under graph isomorphisms. Therefore, this defines an exploration rule. We call this exploration \emph{exploration along the Markov process}.
 
A slight modification of this exploration rule is obtained by fixing $\tau>0$ and letting $X^{1},\dots$ independent $P_G^o$-distributed processes and by exploring new edges along $(X^1_t)_{t\in[0,\tau]}, \,(X^{2}_t)_{t\in[0,\tau]},\,\dots$ consecutively. Meaning that one explores new edges ones one of the paths hits a vertex that has not yet been visited before. We call this exploration \emph{$\tau$-exploration along the Markov process}.
\end{example}

\begin{definition}\label{benjaminischramm} Let $G^ o,G^{o_1}_1,\dots$ be random $\IG^\mathrm{root}$-graphs and let $G$ be connected. Furthermore, let $(K_n)$ be a sequence of  exploration rules and  $(\ell_n)$ be an $\N$-valued sequence. We say that $(G_n^{o_n})_{n\in\N}$  \emph{converges $(K_n,\ell_n)$-locally} to $G^o$ if there exists a coupling of the $K_n$-explorations $\bar{\mathfrak  E}^{(n)}$ of $G_n^{o_n}$ and $\mathfrak E^{(n)}$ of $G^o$ such that
$$
\lim_{n\to\infty} \P(\bar{\mathfrak E}^{(n)}_{\ell_n}\sim \mathfrak E^{(n)}_{\ell_n})= 1.
$$
If one has  local convergence along  every sequence of rules $(K_n)$, then we briefly say that $(G_n^{o_n})_{n\in\N}$  \emph{converges $(\ell_n)$-locally} to  $G^o$.   
\end{definition}

For many network models one can prove local convergence in the sense of Benjamini-Schramm when choosing the root vertex uniformly at random. It is a technical issue to derive $(\ell_n)$-local convergence for sufficiently slowly growing $(\ell_n)$. Often one encounters the following situation, e.g., for the configuration model, inhomogeneous graph model or preferential attachment model: in the limit the graph possesses one giant component (meaning that the relative number of vertices in the largest component is typically close to a deterministic value $p>0$) with all other components being significantly smaller. 
In this case, we consider the restricted graph $G_n^{o_n}$ onto the giant component conditionally on choosing the root $o_n$ from the giant component. We provide technical results that allow to deduce from $(\ell_n)$-local convergence for the unconditional setting the respective $(\ell_n)$-local convergence for the restricted graphs with the root vertex chosen uniformly from the giant component.

\begin{proposition}\label{prop8723} Let $(\ell_n)$ be an $\N_0$-valued sequence converging to infinity.
Let $(G_n^{o_n})_{n\in\N}$ be a sequence of random, finite $\IG^\mathrm{root}$-graphs that converges $(\ell_n)$-locally to a connected random $\IG^\mathrm{root}$-graph $G^o$ and suppose that $\P(G\text{ is infinite})>0$. Then there is an $\N$-valued sequence  $(\rho_n)_{n\in\N}$ that converges monotonically to $\infty$ for which
\begin{align}\label{eq4721}
\liminf_{n\to\infty}  \,\P(\#\cC_n(o_n)\ge \rho_n) \ge \P(G\text{ is infinite}),
\end{align}
where $\cC_n(o_n)$ denotes the component of $o_n$ in $G_n$, and for every such $(\rho_n)$ we have that $G_n^{o_n}$ under $\P_{|\{\cC_n(o_n)\ge \rho_n\}}$ converges $(\ell_n)$-locally to the graph $G^o$ under $\P_{|\{G\text{ is infinite}\}}$.
\end{proposition}

\begin{remark} Let  $(K_n)$ be a sequence of proper exploration rules. Then demanding in Proposition~\ref{prop8723} that $(G_n^{o_n})$ converges $(K_n,\ell_n)$-locally to $G^o$ instead one can deduce as in the proof of the proposition that then $G_n^{o_n}$ under $\P_{|\{\cC_n(o_n)\ge \rho_n\}}$ converges $(K_n,\ell_n)$-locally to the graph $G^o$ under $\P_{|\{G\text{ is infinite}\}}$.
\end{remark}

\begin{proof}Let $K_1,K_2,\dots$ be a sequence of proper exploration rules.
By assumption we can couple the $K_n$-explorations $\bar{\mathfrak E}^{(n)}$ and $\mathfrak E^{(n)}$ of $G_n^{o_n}$ and $G^o$, respectively,  such that $\bar{\mathfrak E}^{(n)}_{\ell_n} \sim {\mathfrak E}^{(n)}_{\ell_n}$, with high probability.
Since the exploration rules are proper we have that if $\ell_n\ge \rho$, on $\{\bar{\mathfrak E}^{(n)}_{\ell_n} \sim {\mathfrak E}^{(n)}_{\ell_n}\}$ one has $\#\cC_n(o_n)\ge \rho$ iff $\#\cC(o)\ge \rho$. Consequently,
$$
\lim_{n\to\infty} \P(\#\cC_n(o_n)\ge \rho) = \P(\# V \ge \rho).
$$
Obviously, as $\rho\to\infty$, we have $ \P(\# V \ge \rho)\to  \P(G\text{ is infinite})$ and existence of an appropriate sequence $(\rho_n)$ follows by a diagonalisation argument. For every sequence  $(\rho_n)$ tending to infinity  one has 
$$
\limsup_{n\to\infty} \P(\#\cC_n(o_n)\ge \rho_n) \le \limsup_{n\to\infty} \P(\#\cC_n(o_n)\ge \rho)= \P(\# V \ge \rho) \to  \P(G\text{ is infinite})
$$
as $\rho\to\infty$ so that~(\ref{eq4721}) entails that indeed
$$
\lim_{n\to\infty} \P(\#\cC_n(o_n)\ge \rho_n) = \P(G\text{ is infinite}).
$$
In order to show the second statement we first  suppose that $\rho_n\le \ell_n$ for all large~$n$. Then one can read off the $K_n$-exploration $\mathfrak E^{(n)}_{\ell_n}$ with $\ell_n$ steps of $G^o$ whether $\#\cC(o)\ge \rho_n$ and the same is true for the respective exploration of the graph $G_n^{o_n}$, say $\bar{\mathfrak E}_{\ell_n}^{(n)}$. Conceiving the explorations $\mathfrak E^{(n)}_{\ell_n}$ and  $\bar{\mathfrak E}_{\ell_n}^{(n)}$  as equivalence classes (modulo graph isomorphisms) we get that
$$
d_\mathrm{TV}\bigl(\P\circ ( \mathfrak E^{(n)}_{\ell_n})^{-1}, \P\circ ( \bar{\mathfrak E}^{(n)}_{\ell_n})^{-1}\bigr)\to 0.
$$
Now generally for two distributions $\cP$ and $\cQ$ and an event $A$ one has for the conditional distributions $\cP_{|A}$ and $\cQ_{|A}$ that for every event $B$
\begin{align*}
\cP_{|A}(B)-\cQ_{|A}(B)&=\frac 1{\cP(A)} \cP(A\cap B)-\frac 1{\cQ(A)} \cQ(A\cap B)\\
&=\frac 1{\cP(A)} (\cP(A\cap B)- \cQ(A\cap B))+ \bigl(\frac 1{\cP(A)}- \frac 1{\cQ(A)}\bigr) \cQ(A\cap B)\\
&\le  \frac 1{\cP(A)} d_\mathrm{TV}(\cP,\cQ)+ \Bigl(\frac {\cQ(A)}{\cP(A)}- 1\Bigr)_+
\end{align*}
With an analogous estimate in the converse direction one thus gets that
$$
|\cP_{|A}(B)-\cQ_{|A}(B)|\le \frac 1{\cP(A)} d_\mathrm{TV}(\cP,\cQ)+\Bigl|\frac {\cQ(A)}{\cP(A)}- 1\Bigr|.
$$
Using this estimate with $A$ being the event that the exploration yields a cluster with more or equal to $\rho_n$ elements yields the estimate
\begin{align*}
d_\mathrm{TV}&\bigl( \P_{|\{\#\cC(o)\ge \rho_n\}}\circ ( \mathfrak E_{\ell_n}^{(n)})^{-1},\P_{|\{\#\cC_n(o_n)\ge \rho_n\}}\circ ( \bar{\mathfrak E}_{\ell_n}^{(n)})^{-1}\bigr)\\
&\le  \frac 1{\P(\#\cC(o)\ge \rho_n)}\, d_\mathrm{TV}(\P\circ ( \mathfrak E_{\ell_n}^{(n)})^{-1},\P\circ (\bar{ \mathfrak E}^{(n)}_{\ell_n})^{-1})+\Bigl|\frac {\P(\#\cC_n(o_n)\ge \rho_n)}{\P(\#\cC(o)\ge \rho_n)}- 1\Bigr|.
\end{align*}
The latter tends to zero by choice of $\rho_n$. Clearly, 
\begin{align}\begin{split}\label{eq3813}
d_\mathrm{TV}&\bigl( \P_{|\{\#\cC(o)\ge \rho_n\}}\circ ( \mathfrak E_{\ell_n}^{(n)})^{-1},\P_{|\{\#\cC(o)=\infty\}}\circ ( { \mathfrak E}_{\ell_n}^{(n)})^{-1}\bigr)\\
&\le d_\mathrm{TV}\bigl( \P_{|\{\#\cC(o)\ge \rho_n\}}\circ (G^o)^{-1},\P_{|\{\#\cC(o)=\infty\}}\circ ( G^o)^{-1}\bigr)\to 0,
\end{split}\end{align}
since $\P(\cC(o)=\infty)>0$.
It remains to consider $n$ with $\rho_n> \ell_n$. First note that applying an arbitrary proper exploration rule for $\ell_n$ steps yields that
\begin{align*}
|\P(\#\cC_n(o_n)\ge \ell_n)- \P(\#\cC(o)=\infty)|&\le  |\P(\#\cC_n(o_n)\ge \ell_n)- \P(\#\cC(o)\ge \ell_n)|\\
&\quad +|\P(\#\cC(o)\ge \ell_n)- \P(\#\cC(o)=\infty)|\to 0.
\end{align*}
Consequently, we get along a subsequence of $n$'s with $\rho_n>\ell_n$ that as in (\ref{eq3813})
\begin{align*}
d_\mathrm{TV}&\bigl( \P_{|\{\#\cC_n(o_n)\ge \rho_n\}}\circ (\bar{ \mathfrak E}_{\ell_n}^{(n)})^{-1},\P_{|\{\#\cC_n(o_n)\ge \ell_n\}}\circ (\bar { \mathfrak E}_{\ell_n}^{(n)})^{-1}\bigr)\\
&\le \frac {\P(\#\cC_n(o_n)\in [\ell_n,\rho_n))}{\P(\#\cC_n(o_n)\ge \ell_n)} =\frac {\P(\#\cC_n(o_n)\ge \ell_n)-\P(\#\cC_n(o_n)\ge\rho_n)}{\P(\#\cC_n(o_n)\ge \ell_n)}.
\end{align*}
The latter tends to zero since by assumption $\P(\#\cC_n(o_n)\ge\rho_n)$ tends also  to $\P(\#\cC(o)=\infty)$.
\end{proof}

\begin{corollary}\label{cor:823757}
Let $(G_n)_{n\in\N}$ be a sequence of random $\IG$-graphs so that every graph $G_n$ has $n$ vertices. Given $G_n$ we choose a root $o_n$ uniformly from the respective set of vertices and we suppose that $(G_n^{o_n})$ converges $(\ell_n)$-locally to a random, connected $\IG^\mathrm{root}$-graph $G^o$.
We let $\cC_n^{\mathrm{max}}$ denote the largest component\footnote{If there is not a unique largest component we choose one uniformly at random.} of $G^o$ and suppose that
$$
\lim_{n\to\infty} \frac{\#\cC_n^\mathrm{max}}{n} =\P(G\text{ is infinite}),\text{ \ in probability}.
$$
We denote by $o_n^\mathrm{max}$ a vertex of  $G_n$ that is, given $G_n$, uniformly distributed on $\cC_n^\mathrm{max}$. Then 
$(G_n^{o_n^\mathrm{max}})_{n\in\N}$ converges $(\ell_n)$-locally to the distribution of $G^o$ under $\P_{|\{G\text{ is infinite}\}}$. 
\end{corollary}

\begin{proof} We set $p= \P(G\text{ is infinite})>0$.
First we show that 
\begin{align}\label{eq8462}
d_\mathrm{TV}(G_n^{o_n^\mathrm{max}}, \P_{|\{o_n\in \cC^\mathrm{max}_n\}}\circ(G_n^{o_n})^{-1})\to 0.
\end{align}
Note that for an event $A$
\begin{align*}
\P(o_n\in \cC_n^\mathrm{max}, G_n^{o_n}\in A)&=\int  \int   \1_{\cC_n^\mathrm{max}}(\bar o_n) \,\1_A (G_n^{\bar o_n})\,  d\mathrm{Unif}_{V_n}(\bar o_n)\,      d\P\\
&=\int  \frac{\#\cC_n^\mathrm{max}} n \int   \1_A (G_n^{\bar o_n})\,  d\mathrm{Unif}_{\cC_n^\mathrm{max}}(\bar o_n)\,      d\P\\
&=\int  \frac{\#\cC_n^\mathrm{max}} n  \1_A (G_n^{o_n^\mathrm{max}})\,      d\P.
\end{align*}
Consequently,
$$
\bigl| \P_{|\{o_n\in \cC^\mathrm{max}_n\}}(G_n^{o_n}\in A)-
\P(G_n^{o_n^\mathrm{max}}\in A)\bigr| \le \int \Bigl| \frac 1{\P(o_n\in\cC^\mathrm{max}_n)}\frac{\#\cC_n^\mathrm{max}} n-1\Bigr|\, d\P
$$
and the right hand side does not depend on the choice of the event $A$. Moreover, it converges  to $0$ as $n\to\infty$ as consequence of dominated convergence, $\frac{\#\cC_n^\mathrm{max}} n\to p$, in probability, and $\P(o_n\in\cC^\mathrm{max}_n)\to p$.

Next, we apply Proposition~\ref{prop8723} with $\rho_n=\frac p2n$: one has
$$
\P(\#\cC_n(o_n)\ge \rho_n)\ge \int \1_{\{o_n\in\cC_n^\mathrm{max}\}}  \1_{\{\#\cC_n^\mathrm{max}\ge \rho_n\}} d\P=\int \frac{\cC_n^\mathrm{max}}{n} \1_{\{\#\cC_n^\mathrm{max}\ge \rho_n\}}\,d\P
$$
and the right hand side converges to $p$ as consequence of dominated convergence. Hence, $G_n^{o_n}$ under $\P_{|\{\cC_n(o_n)\ge \rho_n\}}$ converges $(\ell_n)$-locally to the graph $G^o$ under $\P_{|\{G\text{ is finite}\}}$. Once we showed that 
$$
d_\mathrm{TV} \bigl(\P_{|\{o_n\in \cC_n^\mathrm{max}\}}\circ (G_n^{o_n})^{-1},\P_{|\{\#\cC_n(o_n)\ge \rho_n\}}\circ (G_n^{o_n})^{-1}\bigr)\to 0
$$
we can deduce the statement of the proposition with (\ref{eq8462}) and the previously derived local convergence.
One has
\begin{align*}
d_\mathrm{TV} &\bigl(\P_{|\{o_n\in \cC_n^\mathrm{max}\}}\circ (G_n^{o_n})^{-1},\P_{|\{\#\cC_n(o_n)\ge \rho_n\}}\circ (G_n^{o_n})^{-1}\bigr)\\
& = 1- \frac{\P(o_n\in \cC_n^\mathrm{max},\#\cC_n(o_n)\ge \rho_n)}{\P(o_n\in \cC_n^\mathrm{max})\vee \P(\#\cC_n(o_n)\ge \rho_n)}
\end{align*}
which converges to $0$ since $\P(o_n\in \cC_n^\mathrm{max},\#\cC_n(o_n)\ge \rho_n)$, $\P(o_n\in \cC_n^\mathrm{max})$ and $ \P(\#\cC_n(o_n)\ge \rho_n)$ all tend to $p$.
\end{proof}


%
%
\begin{proposition}\label{prop:631}Let $(m_n)$, $(\ell_n)$ be $\N$-valued sequences and $(\tau_n)$ a $(0,\infty)$-valued sequence. Moreover, let $G_1^{o_1},\dots$ random $\IG^\mathrm{root}$-valued random variables converging $(m_n \ell_n)$-locally to a andom rooted graph $G^o$.
Moreover, suppose that 
$$
m_n \, P^o_G\bigl(\#\mathrm{range}\bigl((X_t)_{t\in[0,2\tau_n]}\bigr)>\ell_n)\to 0,\text{ in probability}.
$$
For $n\in\N$, let $V^{(n)}_0$ be the set of vertices seen by  independent (conditionally on $G^o$) $P^o_G$-processes $X^{(n,1)}=(X^{(1)}_t)_{t\in[0,2\tau_n]},\dots,X^{(n,m_n)}=(X^{(m_n)}_t)_ {t\in[0,2\tau_n]}$ and let $G|_{V^{(n)}_0}$ denote the $V_0^{(n)}$-disclosed subgraph  of $G$. Analogously, let  $\bar V^{(n)}_0$ be the set of vertices seen by the independent $P^{o_n}_{G_n}$-processes $(\bar X^{(n,1)}_t)_{t\in[0,2\tau_n]},\dots,(\bar X^{(n,m_n)}_t)_ {t\in[0,2\tau_n]}$ and let $G_n|_{\bar V^{(n)}_0}$ denote the $V_0^{(n)}$-disclosed subgraph  of $G$.
One can couple 
$$
(G^{o_n}_n|_{\bar V^{(n)}_0}, (\bar X^{(n,1)},\dots, \bar X^{(n,m_n)})) \text{ and } (G^o|_{ V^{(n)}_0},( X^{(n,1)},\dots,  X^{(n,m_n)}))
$$
so that for the coupled random objects, with high probbaility, there is a graph isomorphism $\varphi_n$ taking $G^{o_n}_n|_{\bar V^{(n)}_0}$ to $G^{o}|_{ V^{(n)}_0}$ which also takes the trajectories $\bar X^{(n,1)},\dots, \bar X^{(n,m_n)}$ to $X^{(n,1)},\dots,  X^{(n,m_n)}$.
\end{proposition}

\begin{proof}
For $n\in\N$ we consider  exploration  along independent Markov processes of length $\tau_n$ as introduced in Example~\ref{exa_MP}. Since 
\begin{align*}
P_G^o&(\#\mathrm{range}(X^{(n,1)},\dots , X^{(n,m_n)})>m_n \ell_n)\\
& \le m_n\, P_G^o\bigl(\#\mathrm{range}((X_t)_{t\in[0,2\tau_n]})>\ell_n\bigr) \to 0,\text{ in probability},
\end{align*}
the coupling succeeds with high probability meaning that the disclosed areas and the exploration edges   of the exploration of the first $m_n$ paths can be transformed into each other by an appropriate graph isomorphism, with high probability.
The $k$-th exploration step can be achieved by following the ``current'' trajectory $X^{(n,q)}$ until the next vertex outside of the currently disclosed area is entered. Provided that the coupling succeeded so far the conditional distribution of the trajectory given the next exploration step only depends on the previous exploration step and is invariant under graph isomorphisms. Consequently, the coupling can be extended to a coupling as in the statement of the proposition.  
\end{proof}

\newpage

\section{Markov process estimates}
 
In this section we derive estimates for entrance times that are in the spirit  of Aldous and Brown \cite{Aldous1992}.
We denote by $G\in \b G$  a fixed \emph{connected, deterministic, weighted} graph and we denote by $(P^v:v\in V)$ the reversible continuous time Markov family with generator $L=L_G$ being defined as in~(\ref{def:gen}). Proovided that $\sum_{w\in V} \alpha(w)$ is finite, the distribution
$\pi=\pi^G=(\pi^G_v)_{v\in V}$ given by
$$
\pi_v^G=\frac {\alpha(v)}{\sum_{w\in V} \alpha(w)}\qquad (v\in V)
$$
is the unique invariant distribution.
 We need to introduce some further quantities.
For a  subset $B\subset V$  and $\tau\geq 0$ we call the measure  $
\enter^\tau_{B,G}$ given by
$$
\enter_{B,G}^\tau(y)=\1_B(y) \alpha_y\sum_{z\in B^c} L(y,z)\,  P_G^z(T_B>\tau)\qquad (y\in V)
$$
 the \emph{$\tau$-equilibrium measure of $B$} and let 
$$\capa _{G}^\tau(B) =\|\enter_{B,G}^\tau\|=\sum_{y\in B} \alpha_y\sum_{z\in B^c} L(y,z)\,  P_G^z(T_B>\tau)
$$
the \emph{$\tau$-capacity of $B$}. 
In the particular case that $\tau=0$, we have
$$
\capa_G^0(B)= \sum_{y\in B} \alpha_y\sum_{z\in B^c} L(y,z).
$$



\begin{lemma} \label{le:tec_est} Let $G\in\IG$ be a finite graph, $B\subset V$ and  $y\in V$. One has the following:
\begin{enumerate}
\item  The distribution $P^{\pi}(T_B\in \cdot, X_{T_B}=y)$ has on $(0,\infty)$ the continuous  Lebesgue density
$$t\mapsto \pi_y \sum_{z\in B^c} L(y,z) \,P^z(T_B>t).$$
\item For every $t\ge 2\tau>0$, we have
$$
|P^y(T_B>t)- P^y(T_B>\tau) \, P^\pi(T_B> t)|\leq 2 d_G(\tau)+\sfrac {2\tau}{\alpha(V)} \capa_G^0(B) .
$$
\end{enumerate}
\end{lemma}

\begin{proof}
  (a)
We note that for $y\in B$ as $\delta\downarrow 0$
\begin{align*}
P^\pi(T_B\in (t,t+\delta], X_{T_B}=y)&=P^\pi(T_B\in (t,t+\delta], X_{t+\delta}=y)+o(\delta)\\
&=\pi_y \, P^y((X_s)_{s\in[\delta,t+\delta]}\text{ does not hit $B$}) +o(\delta)\\
&=\delta \pi_y \sum_{z\in B^c} L(y,z)\, P^z(T_B>t) +o(\delta)
\end{align*}
where the $o(\delta)$ term can be controlled uniformly for all $t\ge0$ since   $\max_{x\in V} L_x<\infty$ by finiteness of $V$.
 Hence,
 $$
\lim_{\delta\downarrow 0}  \sup_{t\ge0} \Bigl|\frac 1{\delta} P^\pi(T_B\in (t,t+\delta], X_{T_B}=y)-\pi_y \sum_{z\in B^c} L(y,z)\, P^z(T_B>t)\Bigr|=0.
$$
The term on the right of the minus sign is continuous in $t$ and thus one gets that
the measure $P^\pi(T_B\in \cdot, X_{T_B}=y)$ is continuously differentiable on $(0,\infty)$ with the respective differential. 

(b) 
We note that under $P^\pi$ the expected number of transitions along edges from $B^c$ to $B$ on an interval $I$ is equal to $\frac 1{\alpha(V)} \capa_G^0(B)$ times the length of the interval. Hence,
\begin{align*}
P^\pi(T_B\in (t-2\tau,t])\le P^\pi\bigl(\{(X_t)_{t\in[t-2\tau,t]}&\text{ has a transition along  }\\
&\text{ an edge from $B^c$ to $B$}\}\bigr) \le \sfrac {2\tau}{\alpha(V)} \capa_G^0(B). 
\end{align*}
Now using the Markov property we get with $$\mu_{z,\tau}=P^z(X_\tau\in\,\cdot\,|T_B>\tau)$$
that
\begin{align*}  P^z(T_B>t)&=P^z(T_B>\tau)\,  P^{\mu_{z,\tau}}(T_B>t-\tau)\\
&\le P^z(T_B>\tau) \,P^{\mu_{z,\tau}}\bigl((X_t)_{t\in[\tau,t-\tau]}\text{ does not hit }B\bigr) \\
&\leq P^z(T_B>\tau) (P^\pi(T_B> t-2\tau)+d(\tau))\\
&\leq P^z(T_B>\tau)\Bigl (P^\pi(T_B> t)+\sfrac {2\tau}{\alpha(V)} \capa_G^0(B)+d(\tau)\Bigr).
\end{align*} 
Conversely, 
$$ P^z(T_B>t)\geq  P^z((X_t)_{t\in[0,\tau]\cup[2\tau,t]} \text{ does not hit }B)-   P^{z}(X_\tau \not\in B,(X_t)_{t\in[\tau,2\tau]}\text{ hits } B)
$$
and with $P^{z}(X_\tau\not\in B,(X_t)_{t\in[\tau,2\tau)}\text{ hits } B)\le d(\tau)+\frac {\tau}{\alpha(V)} \capa_G^0(B)$ and
$$
P^z((X_t)_{t\in[0,\tau]\cup[2\tau,t]} \text{ does not hit }B)\ge P^z(T_B>\tau)(P^\pi(T_B> t-2\tau)-d(\tau))
$$
we get that
\begin{align*}
 P^z(T_B>t)&\geq P^z(T_B>\tau)(P^\pi(T_B> t-2\tau)-d(\tau))-d(\tau)-\sfrac {\tau}{\alpha(V)} \capa_G^0(B)\\
 &\ge P^z(T_B>\tau)\,P^\pi(T_B> t)-2d(\tau)-\sfrac {\tau}{\alpha(V)} \capa_G^0(B).
\end{align*}
%
%
\end{proof}

\black

 We will use the previous lemma to show that under appropriate assumptions the distribution of the first entrance time and its location are close in total variation norm to  a product distribution of the exponential and an appropriate discrete distribution. 

\begin{proposition}\label{prop:Aldous} Let $B\subset V$, $\tau>0$, $\lambda=\frac 1{\alpha(V)} \capa^\tau_G (B)$, $\lambda_0=\frac 1{\alpha(V)} \capa^0_G (B)$ \begin{align}\label{eq9873457}
\kappa=\kappa_{G,B,\tau}=4 \frac{\lambda_0}{\lambda } \bigl( d(\tau)+\tau\, \lambda_0\bigr)
 \end{align}
and suppose that $-\log \kappa\ge 2\tau \lambda$ and $2\lambda_0\le \pi(B^c)$, then 
 one has
 \begin{align*}
d_{\mathrm{TV}}&\Bigl(P^\pi_G\circ(T_B,X_{T_B})^{-1}, \mathrm{Exp}(\lambda)\otimes \sfrac1{\capa _{G}^{\tau}(B)}\enter _{B,G}^{\tau}\Bigr)\leq   \pi(B)+2\tau \lambda_0+ \kappa(1+\log 1/\kappa).
\end{align*}
\end{proposition}

\begin{remark}
We note that $\lambda_0\ge \lambda$ so that $2\tau \lambda_0\le \kappa/2$ and the estimate of Proposition~\ref{prop:Aldous} entails that under the same assumptions
 \begin{align*}
d_{\mathrm{TV}}&\Bigl(P^\pi_G\circ(T_B,X_{T_B})^{-1}, \mathrm{Exp}(\lambda)\otimes \sfrac1\lambda\enter_{B,G}^\tau\Bigr)\leq   \pi(B)+  \kappa\bigl (\sfrac 32+\log \sfrac 1\kappa\bigr).
\end{align*}
Moreover, for the conditional distribution $\pi_{B^c}=\pi(B^c)^{-1}\pi|_{B^c}$ we get that
 \begin{align*}
d_{\mathrm{TV}}&\Bigl(P^{\pi_{B^c}}_G\circ(T_B,X_{T_B})^{-1}, \mathrm{Exp}(\lambda)\otimes\sfrac 1\lambda\enter_{B,G}^\tau\Bigr) \leq     \kappa\bigl (\sfrac 32+\log \sfrac 1\kappa\bigr) \pi(B^c)^{-1}.
\end{align*}
\end{remark}

\begin{proof} 
%
  We let $\xi_t=P_G^\pi(T_B>t)$ for $t\geq 0$ and observe that by (a) $\xi_t$ is continuously differentiable since the distribution of $T_B$ has no atoms outside of $0$. Further we get with Lemma~\ref{le:tec_est} (a) that
$$
-\dot \xi_t=\sum_{y\in B} \pi_y\sum_{z\in B^c} L(y,z)  \,P_G^z(T_B>t)
$$
Using (b) we thus obtain with $\eta:=2d(\tau)+2\tau \lambda_0$ that for $t\ge 2\tau$
$$
-\dot \xi_t\leq \lambda \xi_t+ \underbrace{\sum_{y\in B} \pi_y\sum_{z\in B^c}  L(y,z)}_{=\lambda_0} \ \eta.
$$
Moreover,
$$
\xi_{2\tau} \ge \pi(B^c)- 2\tau \lambda_0.
$$

Consequently, $(\xi_t)_{t\geq 2\tau}$ is larger than the solution to the inhomogeneous linear differential equation
$$
-\dot \zeta_t= \lambda \zeta_t +\lambda_0 \,\eta,  \ \ \zeta_{2\tau}=\pi(B^c)-2\tau \lambda_0
$$
whose unique solution is $(\zeta_t)_{t\ge 2\tau}$ given by  
\begin{align*}\zeta_t&=(\pi(B^c)-2\tau \lambda_0) e^{-\lambda (t-2\tau)}-\lambda_0\eta \lambda^{-1}(1-e^{-\lambda (t-2\tau)})\\
& \geq e^{2\tau \lambda} (\pi(B^c)-2\tau \lambda_0) e^{-\lambda t}-\lambda_0\eta /\lambda.
\end{align*}
 Let $\alpha,\beta:(0,\infty)\times B \to [0,\infty)$ denote the densities of $P_G^\pi \circ (T_B,X_{T_B})^{-1}$ (ignoring the atom at zero) and $\mathrm{Exp}(\lambda)\otimes\sfrac1\lambda \enter_{B,G}^\tau$, respectively. Then
\begin{align*}
\mathrm{err}:=&\ d_{\mathrm{TV}}\Bigl(P^\pi_G\circ(T_B,X_{T_B})^{-1},\mathrm{Exp}(\lambda)\otimes\sfrac1\lambda \enter_{B,G}^\tau\Bigr)\\
=&\ 1- \int_0^\infty \sum_{y\in B} \alpha(s,y)\wedge \beta(s,y)\,\dd s.
\end{align*}
Again by Lemma~\ref{le:tec_est}, we get that for $s\ge 2\tau$,  
\begin{align}\begin{split}\label{eq9427}
\alpha(s,y)&\geq \pi_y\sum_{z\in B^c} L(y,z)\bigl(P^z(T_B>\tau)\, \xi_{s}-\eta\bigr)\\
&\geq \pi_y\sum_{z\in B^c} L(y,z)\bigl(P^z(T_B>\tau)\, (e^{2\tau \lambda} (\pi(B^c)-2\tau \lambda_0 e^{-\lambda s} -\lambda_0\, \eta/\lambda)-\eta\bigr)\\
&= (\pi(B^c)-2\tau \lambda_0)e^{2\tau \lambda} \beta(s,y)\\
&\qquad - \pi_y\sum_{z\in B^c} L(y,z)\bigl(P^z(T_B>\tau)\,  \lambda_0\, \eta/\lambda+\eta\bigr).
\end{split} \end{align}
Note that 
$$
\sum_{y\in B} \pi_y\sum_{z\in B^c} L(y,z)\bigl(P^z(T_B>\tau)\,  \lambda_0\, \eta/\lambda+\eta\bigr) \le 2 \lambda_0 \eta =\kappa \lambda.
$$
Letting $\rho:= \bigl((\pi(B^c)-2\tau \lambda_0)e^{2\tau \lambda}\bigr)\wedge1$ we get with~(\ref{eq9427}) that
$$
\sum_{y\in B} \alpha(s,y)\wedge \beta(s,y) \geq  \rho \lambda e^{-\lambda s} -\kappa \lambda.
$$
We set  $t=\lambda^{-1} \log (\rho/\kappa)$ which is by assumption bigger or equal to $2\tau$ and obtain that
and
$$
\mathrm{err}\leq 1-  \int_{2\tau}^t (\rho \lambda e^{-\lambda s}-\kappa \lambda)\,\dd s =1-\rho e^{-2\tau \lambda}+\underbrace{\rho e^{-t\lambda }}_{=\kappa}+\kappa\underbrace{\lambda (t-2\tau)}_{\le \log(1/\kappa)}
$$
We recall that $\lambda\le \lambda_0$ so that by definition of $\rho$
$$
1-\rho e^{-2\tau \lambda}=\max(\underbrace{1-(\pi(B^c)-2\tau \lambda_0}_{=\pi(B)+2\tau \lambda_0}, \underbrace{1-e^{-2\tau \lambda}}_{\le 2\tau\lambda})\le \pi(B)+2\tau \lambda_0
$$
which entails the result.
\end{proof}

In our application we need a combined estimate for the distribution of the path of length $\tau$ and  the next macroscopic entrance event. In the next step, this will allow us to couple the visiting measures  with an appropriate Poisson point process.

\begin{proposition}\label{prop_1241}Under the same assumptions as in Proposition~\ref{prop:Aldous} one has
\begin{align*}
d_{\mathrm{TV}}\bigl(& P^z\circ ((X_t)_{t\in[0,\tau]} , T_B^\tau,X_{T_B^\tau})^{-1},P^z_G\circ((X_t)_{t\in[0,\tau]})^{-1}\otimes \mathrm{Exp}(\lambda) \otimes\sfrac1{ \capa_{G}^\tau(B)} \enter_{B,G}^\tau \bigr)\\
&\le 2 \pi(B)+ \kappa (3+\log 1/\kappa),
\end{align*}
where 
$$
T_B^\tau=\inf\{t\ge\tau:X_t\in B\}.
$$
\end{proposition}

\begin{proof}
We recall that by Proposition~\ref{prop:Aldous}
\begin{align*}
d_{\mathrm{TV}}&\bigl(P^\pi\circ (T_B,X_{T_B})^{-1}, \mathrm{Exp}(\lambda) \otimes \sfrac1\lambda e_{B,G}^\tau\bigr)\\
&\le \pi(B)+2\tau \lambda_0 + \kappa (1+\log 1/\kappa)
\end{align*}
Moreover, for every $z\in V$
$$
d_{\mathrm{TV}}\bigl(P^z\circ (T_B^\tau -\tau,X_{T_B^\tau})^{-1}, P^\pi\circ (T_B,X_{T_B})^{-1}\bigr)\le d(\tau)
$$
so that
\begin{align*}
d_{\mathrm{TV}}\bigl(& P^z\circ ((X_t)_{t\in[0,\tau]} , T_B^{2\tau}-2\tau,X_{T_B^{2\tau}}))^{-1},\\
&P^z\circ ((X_t)_{t\in[0,\tau]})^{-1} \otimes   P^\pi\circ (T_B,X_{T_B})^{-1}\bigr)\le d(\tau).
\end{align*}
Since 
\begin{align*} P^z&((T_B^{2\tau},X_{T_B^{2\tau}})\not=(T_B^\tau  ,X_{T^\tau_B}))= P^z((X_t)_{t\in[\tau,2\tau)}\text{ hits } B)\\
& \le d(\tau)+\pi(B)+\tau \lambda_0
\end{align*}
we also get that
\begin{align*}
d_{\mathrm{TV}}\bigl(& P^z\circ ((X_t)_{t\in[0,\tau]} , T_B^\tau-2\tau,X_{T_B^\tau})^{-1},\\
&P^z\circ ((X_t)_{t\in[0,\tau]})^{-1} \otimes   P^\pi\circ (T_B,X_{T_B})^{-1}\bigr)\le 2d(\tau)+\pi(B)+\tau \lambda_0.
\end{align*}
Altogether we thus get that
\begin{align*}
d_{\mathrm{TV}}\bigl(& P^z\circ ((X_t)_{t\in[0,\tau]} , T_B^\tau -2\tau,X_{T_B^\tau})^{-1},P^z\circ ((X_t)_{t\in[0,\tau]})^{-1} \otimes \mathrm{Exp}(\lambda) \otimes \sfrac1\lambda e_{B,G}^\tau\bigr)\\
&\le 2d(\tau)+2 \pi(B)+3\tau \lambda_0 + \kappa (1+\log 1/\kappa)
\end{align*}
Shifting the exponential distribution $\mathrm{Exp}(\lambda)$  by $2\tau$ time units causes another perturbation that is bounded by $2\tau\lambda\le 2\lambda_0)$ and we finally get that
\begin{align*}
d_{\mathrm{TV}}\bigl(& P^z\circ ((X_t)_{t\in[0,\tau]} , T_B^\tau,X_{T^\tau_B})^{-1},P^z\circ ((X_t)_{t\in[0,\tau]})^{-1} \otimes \mathrm{Exp}(\lambda) \otimes \sfrac1\lambda\enter_{B,G}^\tau \bigr)\\
&\le 2d(\tau)+2 \pi(B)+5\tau Q(B) + \kappa (1+\log 1/\kappa).
\end{align*}
\end{proof}

\begin{proposition}\label{prop8743} Let $G\in\IG$ be a finite graph and let $\tau>0$, $B\subset V$ and $a>0$. We  denote by $\Xi$ the $(B,\tau, a)$-local visits  of an $L_G$-process $X$ started in equilibrium. Further let 
$\kappa$ be as in Prop.~\ref{prop:Aldous}.

We denote by  $\Gamma$  the Poisson point process on $[0,\infty)\times W_+$ with intensity measure 
$$ \frac{a}{\alpha(V)}\, \mathrm{Leb}|_{[0,\infty)}\otimes \bigl( P^{\enter_{B,G}^{\tau}}_G\circ ((X_t)_{t\in[0,\tau]})^{-1}\bigr).$$ 
Then for every $t\geq 0$ and $\rho\in\N$
$$
d_{\mathrm{TV}} (\Xi|_{[0,t]\times W_+}, \Gamma|_{[0,t]\times W_+})\leq c(G,B,\tau) \,\rho+2^{-\rho} e^{a\capa_G^\tau(B) t/\alpha(V)},
$$
where
$$
c(G,B,\tau)=2 \pi(B)+ \kappa (3+\log 1/\kappa).
$$
\end{proposition}

The proof is based on the following lemma whose proof is straight-forward and  left to the reader.
\begin{lemma}\label{le0346}
For $i=1,2$ let  $A_i$ and $B_i$ be random variables taking values in the same Polish space.
Then 
$$
d_{\mathrm{TV}}(\P_{A_1,A_2},\P_{B_1,B_2}) \leq d_{\mathrm{TV}}(\P_{A_1},\P_{B_1}) + \E[d_{\mathrm{TV}}(\P_{A_2|A_1}(A_1,\cdot), \P_{B_2|B_1}(A_1,\cdot))],
$$
where $\P_{A_2|A_1}$ and $\P_{B_2|B_1}$ are arbitrary regular conditional distributions.
\end{lemma}

\begin{proof}[Proof of Proposition \ref{prop8743}] As in Section~\ref{sec_visit}, we denote by $T_1,\dots$ and $Y_1,\dots$ the $(B,\tau)$-entrance times and $(B,\tau)$-visits of $X$, respectively. 
We control the total variation distance between the distributions of 
$$(T_1,Y_1,\dots, Y_{\rho-1},T_\rho-T_{\rho-1})$$ and $$(S_1,Z_1,\dots,S_{\rho-1}, Z_{\rho-1},S_\rho ),$$ where $S_1,\dots$ and $Z_1,\dots$ are independent random variables, the former ones being $\mathrm{Exp}(a\capa^\tau_G(B)/\alpha(V))$-distributed and the latter ones being $P^{\frac 1{\capa_G^\tau(B)}\enter_{B,G}^{\tau}}_G\circ ((X_t)_{t\in[0,\tau]})^{-1}$-distributed processes.
It is more convenient to control the total variation distance of the extended vectors  $${\bf Y}^\rho:=(T_1,Y_1(0),Y_1,T_2-T_1,Y_2(0),\dots,Y_{\rho-1},T_\rho-T_{\rho-1},Y_\rho(0))$$ and  
$${\bf Z}^\rho:=(S_1,Z_1(0),Z_1,S_2,Z_2(0),\dots,Z_{\rho-1},S_\rho,Z_\rho(0)).
$$
We choose $A_1={\bf Y}^{\rho-1}$, $A_2=(Y_{\rho-1},T_{\rho}-T_{\rho-1}, Y_{\rho}(0))$, $B_1={\bf Z}^{\rho-1}$ and $B_2=(Z_{\rho-1},S_{\rho}, Z_{\rho}(0))$ and apply Proposition~\ref{prop_1241} together with Lemma~\ref{le0346} to infer that
$$
d_{\mathrm{TV}}({\bf Y}^\rho,{\bf Z}^\rho)\leq d_{\mathrm{TV}}({\bf Y}^{\rho-1},{\bf Z}^{\rho-1})+ c(G,B,\tau).
$$
Iteration of this argument yields together with an application of Proposition~\ref{prop:Aldous} that
$$
d_{\mathrm{TV}}({\bf Y}^\rho,{\bf Z}^\rho)\leq d_{\mathrm{TV}}((T_1,Y_1(0)), (S_1,Z_1(0)))+ (\rho-1) c(G,B,\tau)\leq \rho \,c(G,B,\tau).
$$
Note that  the Poisson point process $\Gamma$ has the same distribution as
$$
\sum_{k=1}^\infty \delta_{(S_1+\ldots+S_k, Z_k)}
$$
 so that for positive $\theta$
\begin{align*}
d_{\mathrm{TV}}(\Xi|_{[0,t]\times D_+},& \Gamma|_{[0,t]\times D_+})\leq d_{\mathrm{TV}}({\bf Y}^\rho,{\bf Z}^\rho) +\P(S_1+\ldots+S_{\rho}\leq t)\\
&\leq  m c+\E[e^{-\theta  S_1}]^\rho/e^{-\theta t}=\rho C +\Bigl(\frac{1}{1+\alpha(V)\theta/(a\capa_G^\tau(B)})\Bigr)^\rho e^{\theta t}.
\end{align*}
Choosing $\theta= a\capa_G^\tau(B)/\alpha(V)$ yields the result.
\end{proof}

\section{Proof of the main result}

As we have seen in Proposition~\ref{prop8743} that under appropriate mixing conditions the visiting measure is  related to a Poisson point process. In the view of local approximation the equilibrium measures appearing  in the intensity measures are in a certain sense global properties. In oder to establish coupling we need to replace these quantities by quantities that can be derived via exploration.

We introduce the concept for a fixed  general finite graph $G^o\in\IG^\mathrm{root}$. Later we will apply the results on the sequence of random finite graphs.

\begin{construction}\label{constr1}Beyond a fixed rooted graph $G^o\in\IG^\mathrm{root}$ the construction depends on four parameters $m,R\in\N$ and $\tau,\mathfrak a>0$.

 Let   $ X^{(1)}=(X^{(1)}_t)_{t\in[0,2\tau]}$, $X^{(2)}=(X^{(2)}_t)_{t\in[0,2\tau]},\dots$ be independent $P^{o}_{G}$-distributed processes  (killed at time $2\tau$).
 Based on the trajectories we build two stacks of paths. First we use the first $m$ trajectories $ X^{(1)},\dots, X^{(m)}$ to build stacks labeled by the vertices in $\partial V^{o}[R+1]$. 
For this we briefly write $ T_r^{(k)}$ $(r\in\{R,R+1\})$ for the first hitting time of $ X^{(k)}$ of $\partial V^{o_n}[r]$ and put for every $k=1,\dots,m$ with $T_{R+1}^{(k)}\le \tau$ the  path $( X^{(k)}_{ T^{(k)}_{R+1}+t })_{t\in[0,\tau]}$  on the $X^{(k)}_{ T^{(k)}_{R+1}}$-stack.
The second stack makes use of the trajectories from index $m+1$ on. We iteratively put  for every $k=m+1,\dots$ with $ T_{R}^{(k)}\le \tau$ the  path $(X^{(k)}_{ T^{(k)}_{R}+t })_{t\in[0,\tau]}$  on the $ X^{(k)}_{ T^{(k)}_{R}}$-stack.

Note that by construction the number of trajectories on the first kind of stacks forms a multinomially-distributed random variable with $m$-draws and success probabilities 
$$P^{o}_{G} (T_{\partial V_n^{o}[R+1]}\le \tau, X_{T_{\partial V^{o}[R+1]}}=v).
$$
Moreover,  every stack $v\in \partial V^{o}[R]$ (of the second kind) has infinitely many paths since
$P^{o}_{G} (T_{\partial V_n^{o}[R]}\le \tau, X_{T_{\partial V^{o}[R]}}=v)$ is strictly positive. 
Given the  sizes of the  stacks of the first kind, all stacks  contain independent $L_G$-Makrov processes on $[0,\tau]$ started in the label of the respective stack.

We declare every trajectory on one of the stacks $v$ of first kind as $v$-escape if the trajectory does not hit $V^o[R]$ and we denote by  $\mathfrak s_v$ the relative number of $v$-escapes  on the $v$-stack (with the convention that  $\mathfrak s_v=0$, if the stack is empty). 
We let, for every vertex $v\in V$,
$$
\mathfrak e_{V^o[R],G}^{\tau}(v)=\1_{V^o[R]} (v)\,\alpha (v) \sum_{w\in \partial V^o[R+1]} L_{v,w}^G \mathfrak s_w.
$$
We introduce a doubly stochastic Poisson point process (Cox process) $\gamma$ on $[0,\infty)\times \partial V^o[R]$ that is conditionally on the first $m$ trajectories a Poisson point process with density $\mathfrak a^{-1}\mathrm{Leb}|_{[0,\infty)}\otimes \mathfrak e_{V^o[R],G}^\tau$ and is independent of the paths $m+1,\dots$. Next, we construct a point process 
$\Gamma$ on $[0,\infty)\times W_+$ by ordering for each $v\in \partial V^o[R]$ the epochs of $\gamma$ on $[0,\infty)\times \{v\}$ increasingly and adding for the $k$-th point, say at time $s_k$, the tuple consisting of $s_k$ and the $k$-th trajectory of the $v$-stack to $\Gamma$. By construction the trajectories on the stacks are independent of $\gamma$ and we get that $\Gamma$ is a Cox-process that is, given the first $m$ trajectories, a Poisson point process with random intensity measure
$$
\mathfrak a^{-1}\mathrm{Leb}|_{[0,\infty)}\otimes P^{ \mathfrak e_{V^o[R],G}^\tau}\circ ((X_t)_{t\in[0,\tau]})^{-1}.
$$  
 \end{construction}

\begin{proposition}\label{prop:main}
Let $G^o\in\IG$ be a finite graph, $m,R\in\N$  and $\tau ,\mathfrak a, a,\sigma>0$. We assume the assumptions of Proposition~\ref{prop:Aldous} and choose $c_{G,V^o[R],\tau}$  as in Proposition~\ref{prop8743}.
We denote by $\Xi$ the $(V^o[R],\tau,a)$-visiting measure of a $P_G$-Markov process started in equilibrium and by $\Gamma$ the Cox-process constructed in Construction~\ref{constr1}. Then one has
\begin{align*}
d_{\mathrm{TV}}&(\Xi|_{[0,\sigma]\times W_+},\Gamma|_{[0,\sigma]\times W_+})\\
 \le&\,  \E\Bigl[1\wedge \bigl( c_{G,V^o[R],\tau} \rho +2^{-\rho} e^{a \capa_G^\tau(V^o[R]) \sigma/\alpha(V)}+ \sigma \bigl\|\mathfrak a^{-1} \mathfrak e_{V^o[R],G}^\tau- \sfrac a{\alpha(V)} e_{V^o[R],G}^\tau \bigr\|_1\bigr)\Bigr],
\end{align*}
where $\|\cdot\|_1$ denotes the $\ell^1$-norm and
\end{proposition}

\begin{proof} 
By Proposition~\ref{prop8743} we have that
$$
d_{\mathrm{TV}}(\Xi|_{[0,\sigma]\times W_+},\bar \Gamma|_{[0,\sigma]\times W_+})\le  c_{G,B,\tau} m +2^{-m} e^{a \capa_G^\tau(V^o[R]) t/\alpha(V)},
$$
where $\bar \Gamma$ is the Poisson point process as in the latter proposition. Recalling that~$\Gamma$ is a Cox-process that, given the first $m$ trajectories $X^{(1)},\dots,X^{(m)}$, is a Poisson point process with  intensity measure
$$
\mathfrak a^{-1}\mathrm{Leb}|_{[0,\infty)}\otimes P^{ \mathfrak e_{V^o[R],G}^\tau}\circ ((X_t)_{t\in[0,\tau]})^{-1}
$$  
we conclude that 
$$
d_{\mathrm{TV}}(\bar \Gamma|_{[0,\sigma]\times W_+},\Gamma|_{[0,\sigma]\times W_+})\le  \E\Bigl[1\wedge \bigl( \sigma \bigl\|\mathfrak a^{-1} \mathfrak e_{V^o[R],G}^\tau- \sfrac a{\alpha(G)} e_{V^o[R],G}^\tau \bigr\|_1\bigr)\Bigr].
$$
The triangle inequality yields the statement of the proposition.
\end{proof}

\begin{remark}
Applying a diagonalization argument one can show that under the same assumptions as in the previous proposition there exist monotone $\N$-valued sequences $(R'_n)_{n\in\N}$ and $(\sigma_n)_{n\in\N}$ tending to infinity such that the statement of the theorem still holds when replacing $R'$  by $R'_n$ and $\sigma$ by $\sigma_n$.
\end{remark}

\begin{proof}[Proof of Theorem~\ref{prop:2746}]
To establish the wanted coupling we do various couplings in between.

First let $\Gamma$ denote a Cox-process with points in $[0,\infty)\times W_+$ that is, given $G^o$, a Poisson point process with intensity measure
$$
\frac 1{\mathfrak a} \mathrm{Leb}|_{[0,\infty)}\otimes P^{e_{V^o[R],G}}\circ ((X_t)_{t\in [0,\infty)})^{-1}.
$$
We will apply explorations as in Example~\ref{exa_MP} that depend on an index $n\in\N$. We choose an increasing $\N$-valued sequence $(m_n)_{n\in\N}$ tending to infinity with $m_n \ell_n'\le \ell_n$ and 
$$
m_n P^{o}_G (\#\mathrm{\, range}(X_t)_{t\in[0,2\tau_n]}>\ell_n') \to0, \text{ in probability}.
$$

 For $n\in\N$ we let $K_n$ be the $2\tau_n$-exploration along the Markov process and we denote by $X^{(n,1)}=(X^{(n,1)}_t)_{t\in[0,2\tau_n]},  X^{(n,2)}=( X^{(n,2)}_t)_{t\in[0,2\tau_n]},\dots$ \cadlag\ processes that are, given $G_n$, independent $P^{o}_{G}$-distributed processes (constrained on the time interval  $[0,2\tau_n]$) and we denote by ${\mathfrak E}_{\ell_n}^n$ the respective $K_n$-exploration  after $\ell_n$ steps.

We apply Construction~\ref{constr1} with parameters $m_n,R,\tau_n,\mathfrak a$ for the processes $X^{(n,1)}$, $X^{(n,2)},\dots$ and let $ {\mathfrak e}^{(n)}$ be  the estimate of the equilibrium measure $e^{\tau_n}_{V^o[R],G}$ based on the first $m_n$ trajectories $X^{(n,1)},\dots,X^{(n,m_n)}$ as introduced in the construction. Moreover, we denote by $\gamma_n$ and  $\Gamma_n$ the respective point processes.
Then given $X^{(n,1)},\dots , X^{(n,m_n)}$, $\Gamma_n$ is a Poisson point process with intensity
$$
\frac 1{\mathfrak a} \mathrm{Leb}|_{[0,\infty)}\otimes P^{\mathfrak e^{(n)}}\circ ((X_t)_{t\in [0,\tau_n]})^{-1}.
$$ 
Consequently,
\begin{align*}
d_{\mathrm{TV}}\bigl(&(G^o,\Gamma|_{[0,\sigma]\times W_+}\circ \mathrm{conf}_{[0,\tau_n]}^{-1}),(G^o,\Gamma_n|_{[0,\sigma]\times W_+})\bigr)\\
&\qquad \le \E\bigl[1\wedge \bigl(\sigma \|e_{V^o[R],G}-\mathfrak e^{(n)}\|_{\ell_1}\bigr)\bigr]
\end{align*}
We have a close look at the estimate $\mathfrak e^{(n)}$. Given $G^o$ it is the weighted sum of the estimates of the escape probabilities. Since $m_n,\tau_n\to\infty$ we can apply the law of large numbers (given $G^o$) and obtain that $\|\mathfrak e_{V^o[R],G}^{\tau_n}-e_{V^o[R],G}^{\tau_n}\|_{\ell_1}\to 0$, in probability. Moreover, by monotone convergence also $\|e_{V^o[R],G}^{\tau_n}-e_{V^o[R],G}\|_{\ell_1}\to0$, almost surely. Hence, we can couple $(G^o,\Gamma)$ and $(G^o,\Gamma_n)$ such that, with high probability, $(G^o,\Gamma|_{[0,\sigma]\times W_+}\circ \mathrm{conf}_{[0,\tau_n]}^{-1})=(G^o,\Gamma_n|_{[0,\sigma]\times W_+})$.

Next, we also apply Construction~\ref{constr1} for every $n\in\N$ onto  the graph $G^{o_n}_n$, again with parameters $m_n,R,\tau_n,\mathfrak a$: given $G_n^{o_n}$  we denote by $\bar X^{(n,1)}=(\bar X^{(n,1)}_t)_{t\in[0,2\tau_n]}$, $\bar X^{(n,2)}=(\bar X^{(n,2)}_t)_{t\in[0,2\tau_n]},\dots$ independent $P^{o_n}_{G_n}$-distributed processes and by $\bar\gamma_n$ and $\bar \Gamma_n$ the respective Cox-process defined as in the construction on the basis of the trajectories $\bar X^{(n,1)},\dots$. 
 We denote by $K_n$ the exploration rule associated to $2\tau_n$-exploration along the Markov process as introduced in Example~\ref{exa_MP} and write for the induced $\ell_n$-step $K_n$-explorations $\mathfrak E_{\ell_n}^{(n)}$ and $\bar{\mathfrak E}_{\ell_n}^{(n)}$ when applied to $G^o$ along the paths $X^{(1,n)},\dots$ and to $G^{o_n}_n$ along the paths $\bar X^{(1,n)},\dots$, respectively.  We denote by $\mathcal V_n$ and $\bar {\mathcal V}_n$ the disclosed areas under $\mathfrak E_{\ell_n}^{(n)}$ and $\bar{\mathfrak E}_{\ell_n}^{(n)}$ and note that, with high probability, 
  there exists an isomorphism $\varphi_n$ taking the $\bar{\mathcal V}_n$-disclosed graph of $G_n^{o_n}$ to the $\mathcal V_n$-disclosed graph of $G^o$. If, additionally, the $2m_n$ trajectories $X^{(n,1)},\dots, X^{(n,2m_n)}$ have combined range less than or equal to $\ell_n$, then the latter trajectories are obtained by applying $\varphi_n$ onto the paths $\bar X^{(n,1)},\dots, \bar  X^{(n,2m_n)}$. 
  This happens with high probability, 
  since 
  $$
\P(\#\mathcal V_n>\ell_n) \le  \E\bigl[\bigl(m_n P^o_G \bigl(\#\, \mathrm{range}\bigl((X_t)_{t\in[0,2\tau_n]}\bigr)> \ell_n'\bigr)\bigr)\wedge1\bigr]\to 0
$$
   In the latter case, in particular, $\varphi_n$ takes $\bar{\mathfrak e}^{(n)}$ to ${\mathfrak e}^{(n)}$ and we can couple $\gamma_n$ and $\bar \gamma_n$ in such a way that $\varphi_n$ takes $\bar \gamma_n$ to $\gamma_n$.
Note that $m_n\to\infty$ implies that the minimal  number (taken over all boundary vertices $v\in\partial V^o[R]$) of trajectories with index in $m_n+1,\dots 2m_n$ starting in a boundary vertex tends to infinity, in probability, so that with high probability $\Gamma_n|_{[0,\sigma]\times W_+}$ can be constructed from $\gamma_n|_{[0,\sigma]\times W_+}$ on the basis of the trajectories $m_n+1,\dots ,2m_n$. We thus get that, with high probability, $\varphi_n$ takes $\bar \Gamma_n|_{[0,\sigma]\times W_+}$ to $\Gamma_n|_{[0,\sigma]\times W_+}$.

It remains to couple $(G_n^{o_n},\Gamma_n)$ and $(G_n^{o_n},\Xi_n)$.
 As consequence of Proposition~\ref{prop:main}, we have that
\begin{align}\label{coup_err1}
d_{\mathrm{TV}}((G_n^{o_n},\Xi_n|_{[0,\sigma]\times W_+}),(G_n^{o_n},\bar \Gamma_n|_{[0,\sigma]\times W_+}))\le\E\bigl[1\wedge \bar \delta_n\bigr],
\end{align}
where
\begin{align*}
\bar \delta_n&= c_{G_n,V_n^{o_n}[R],\tau_n} \rho +2^{-\rho} e^{a_n \capa_{G_n}^{\tau_n}(V_n^{o_n}[R]) \sigma/\alpha_n(V)}\\
& \qquad + \sigma \bigl\|\mathfrak a^{-1} \mathfrak e_{V_n^{o_n}[R],G_n}^{\tau_n}- \sfrac {a_n}{\alpha_n(V_n)} e_{V_n^{o_n}[R],G_n}^{\tau_n} \bigr\|_{\ell_1},
\end{align*}
if 
\begin{align}\label{eqa83}-\log \kappa_{G_n,V_n^{o_n}[R],\tau_n}\ge 2\tau_n \frac{\capa_{G_n}^{\tau_n}(V_n^{o_n}[R])}{\alpha_n(V_n)}\text{   \ and \  }2\capa^0_{G_n}(V^{o_n}_n[R])\le \alpha_n(V^{o_n}_n[R]^c),
\end{align}
  and $\bar\delta_n=1$, otherwise.
It remains to verify that the right hand side of (\ref{coup_err1}) tends to zero. 
First we show that $\|\mathfrak e^{\tau_n} _{V_n^{o_n}[R],G_n}-e_{V_n^{o_n}[R],G_n}\|_{\ell_1}$ tends to zero, in probability. Note that with high probability the exploration used in the coupling of $\bar \Gamma_n$ and $\Gamma_n$ unveils the $R'$-neighbourhood of the root so that, with high probability, $\varphi_n$ takes
the $R'$ neighbourhood of $G^{o_n}_n$ to the one of $G^o$.
Moreover, since $m_n\to\infty$  the minimal number of trajectories on one of the stacks of first kind tends to infinity, in probability. Since given the stack sizes the estimates for the equilibrium measure is unbiased as long as the stack is nonempty we get that
$$
\|\mathfrak e^{\tau_n} _{V_n^{o_n}[R],G_n}-e_{V_n^{o_n}[R],G_n}\|_{\ell_1}\to 0, \text{ \ in probability}.
$$
Since $\alpha_n(V_n)/a_n\to \mathfrak a>0$, in probability, we also get that
\begin{align}\label{eq82472}
\Bigl\|\mathfrak a^{-1}\mathfrak e^{\tau_n} _{V_n^{o_n}[R],G_n}-\frac {a_n}{\alpha_n(V_n)}e_{V_n^{o_n}[R],G_n}\Bigr\|_{\ell_1}\to 0, \text{ \ in probability.}
\end{align}
Moreover, we have, in probability,
\begin{align*}
\pi_n(V^{o_n}_n[R])& = \frac {\alpha_n(V^{o_n}_n[R])}{\alpha_n(V_n)}\to 0 \text{ \ \ and}\\
 \kappa_{G_n,V^{o_n}_n[R],\tau_n} & = 4\frac {\mathrm{cap}_{G_n}^0(V^{o_n}_n[R])}{\mathrm{cap}_{G_n}^0(V^{o_n}_n[R])}\Bigl(d_{G_n}(\tau_n)+\tau_n \frac{ \capa_{G_n}^{\tau_n}(V^{o_n}_n[R])}{\alpha_n(V_n)}\Bigr)\to0,
\end{align*}
where we used that $$\alpha_n(V^{o_n}_n[R]),\ \1_{\{\mathrm{cap}_{G_n}^0(V^{o_n}_n[R])\not=0\}}\frac {\mathrm{cap}_{G_n}^0(V^{o_n}_n[R])}{\mathrm{cap}_{G_n}^0(V^{o_n}_n[R])} \text{ \ and \ } \capa_{G_n}^{\tau_n}(V^{o_n}_n[R])
$$ are tight (as consequence of local convergence) and $\alpha_n(V_n)\to \infty$, $\tau_n/\alpha_n(V_n)\to 0$ and $d_{G_n}(\tau_n)\to 0$, in probability, by assumption. 
Consequently, $c_{G_n,V_n^{o_n}[R],\tau_n}\to0$, in probability.
We note that~(\ref{eqa83}) is satisfied with high probability, since $\kappa_{G_n,V^{o_n}_n[R],\tau_n}\to0$, $\tau_n /\alpha_n(V_n)\to0$ and  $\alpha_n(V_n^{o_n}[R]^c)\to \infty$, in probability.
Hence, we can choose $\rho$ in dependence on $n$ such that $\rho_n\,c_{G_n,V_n^{o_n}[R],\tau_n}\to0$, in probability, and obtain that $\bar \delta_n\to 0$, in probability, by tightness of $(e^{a_n\capa_{G_n}^{\tau_n}(V_n^{o_n}[R])\sigma/\alpha_n(V)}:n\in\N)$ and convergence~(\ref{eq82472}).

Consequently, we can also couple  $(G_n^{o_n},\bar\Gamma_n)$ and $(G_n^{o_n},\Xi_n)$ such that, with high probability, $\bar\Gamma_n|_{[0,\sigma]\times W_+}=\Xi_n|_{[0,\sigma]\times W_+}$. As observed above, with high probability, $\mathcal V_n$ and $\bar{\mathcal V}_n$ contain the respective $R'$-neighbourhoods and applying the isomorphism $\varphi_n$ between the $\bar{\mathcal V}_n$- and $\mathcal V_n$-disclosed graphs of $G^{o_n}_n$ and $G^o$ onto the labels in $\Xi_n|_{[0,\sigma]\times W_+}$ yields $\Gamma|_{[0,\sigma]\times W_+}\circ \mathrm{conf}_{[0,\tau_n]}^{-1}$.
\end{proof}

\begin{proof}We choose an $\N$-valued sequence $(m_n)_{n\in\N}$ tending to infinity such that  $m_n \ell_n'\le \ell_n$ for large $n$ and
$$
 \lim_{n\to\infty} m_n \,P_G^o(\# \text{ disc.\ of }(X_t)_{[0,2\tau_n]} >\ell_n')   =0,\text{ in probability.}
$$

We define an $n$-dependent exploration rule $K_n$ as follows. Starting from the root we explore the graph along the Markov process on chain  for $\ell_n'$ steps  (meaning that the exploration follows the embedded Markov chain). 
Once the first path, say $(Y_1^{(n)}(k))_{k=0,\dots,\ell_n'}$, has been generated, we generate another path, say $(Y^{(n)}_2(k))_{k=0,\dots,\ell_n'}$, in the same manner starting again from the root vertex. We continue and generate further independent  paths $(Y^{(n)}_3(k))_{k=0,\dots,\ell_n'}$, $(Y^{(n)}_4(k))_{k=0,\dots,\ell_n'}$, \dots. Obviously, this defines an exploration rule and the related picked edges are the ones that induce the respective transitions of the embedded Markov chain.

Since $m_n\ell_n'\le \ell_n$ for large $n$ we can couple the rooted graphs and their explorations in such a way that,  with high probability, the $K_n$-exploration $\mathfrak E_{\ell_n}^{(n)}$ of $G^o$ of length $\ell_n$  is equal to $\varphi_{n}(\bar {\mathfrak E}_{\ell_n}^{(n)})$, where $\bar{ \mathfrak E}_{\ell_n}^{(n)}$ is a $K_n$-exploration of $G_n^{o_n}$ of length $\ell_n$ and $\varphi_n$ is an appropriate random graph isomorphism between the explorations. 
 In the  case  that the latter identity holds we say that the $n$-th coupling succeeds. We write $(Y_k^{(n,1)})_{k=0,\dots,\ell_n'}$, $(Y^{(n,2)}_k)_{k=0,\dots,\ell_n'}$,\dots for the chains of vertices visited by the $n$-th exploration of $G^{o}$ and $(\bar Y^{(n,1)}_k)_{k=0,\dots,\ell_n'}$, $(\bar Y_k^{(n,2)})_{k=0,\dots,\ell_n'}$,\dots for the chains of vertices visited by the $n$-th exploration of $G_n^{o_n}$. With high probability,
 the explorations $\mathfrak E_{\ell_n}^{(n)}$, resp.\ $\bar {\mathfrak E}_{\ell_n}^{(n)}$, disclose the $R'+1$ neighbourhood of $o$ in $G$, resp.\ of $o_n$ in $G_n$  since $m_n,\ell_n'\to \infty$.
 
 Next, we relate the embedded chains $Y^{(n,j)}$ and $\bar Y^{(n,j)}$ to continuous time paths: on an appropriately rich probability space we can assume existence of  independent standard exponentials $(\rho_{j,k})_{j,k\in\N}$ (independent of  all the random objects considered so far) and we let for $n,j\in\N$ and $k=1,\dots, \ell_n'+1$,
 $$
\tau_0^{(n,j)}=0\text{ \ and \ }  \tau_k^{(n,j)}=\tau_{k-1}^{(n,j)}+ \alpha_G(Y^{(n,j)}_{k-1})^{-1}  \rho_{j,k}
 $$
 and set for $t\in [ \tau_{k-1}^{(n,j)}(k-1), \tau_k^{(n,j)})$
 $$
 \cX^{(n,j)}_t= Y_{k-1}^{(n,j)}
 $$
 and for completeness $\cX^{(n,j)}_t=\partial$ if $t\ge  \tau^{(n,j)}_{\ell_n'+1}$.
 Analogously, we define $(\bar \cX^{(n,j)}_t)_{t\ge0}$ by replacing in the above definitions $\alpha_G$ by $\alpha_{G_n}$ and $Y_j^{(n)}$ by $\bar Y_j^{(n)}$.
 Note that in the case that the coupling succeeds one recovers the paths $\cX^{(n,1)},\dots, \cX^{(n,1)}$ from $\bar\cX^{(n,1)},\dots, \bar \cX^{(n,1)}$ by applying the isomorphism $\varphi_n$.

We construct an estimate for $P_G^z(T_{V^o[R]}=\infty)$.
Although the following definitions all  depend on the index $n$ of the exploration  of $G^o$ we omit this in the notation. 
For $v\in V$ we call  $j\in\{1, \dots,m_n\}$ a \emph{$v$-indicator}, if 
$(\cX^{(n,j)}_t)_{t\in[0,2\tau_n]}$ hits~$v$ before time $\tau_n$ and we call it a \emph{$v$-escape} if, additionally, the latter path does not return to $V^o[R]$ for $\tau_n$ time units after its first entrance  into~$v$.
We let for $v\in V$, $\mathfrak e^{(n)}_v$ denote the relative number of $v$-escapes, i.e.,
$$
 \mathfrak e^{(n)}_v=\begin{cases}\frac {\# \{v\text{-escapes}\} }{\#\{ v\text{-indicators}\}}, &\text{ if } \{v\text{-indicators}\}\not=\emptyset,\\ 0, &\text{ else.}\end{cases}
$$

Next, we show that typically for all $v\in \partial V^o[R+1]$, $\mathfrak e_v^{(n)}$ is for large $n$ close to $P^v_G((X_t)\text{ does not hit }V ^o[R])$.
By construction, one has, conditionally on $G^o$, that the chains of the $K_n$ explorations are independent  and the respective continuous time processes $(\cX_t^{(n,1)})_{t\in[0,2\tau ]},\dots$ behave like $P^{o}_G$-Markov processes with the modification that the $\ell_n'+1$-th jump leads to the cemetery state $\partial$. On an appropriate probability space, one can define \cadlag\ processes $(X^{(n,1)}_t)_{t\ge0 } ,  (X^{(n,2)}_{t})_{t\in[0,\infty)},\dots$ so that given $G^o$ the processes are independent $P_G^o$-Markov processes that satisfies for every $t\ge0$, 
$X_t^{(n,j)}=\cX^{(n,j)}_t$ iff $X^{(n,j)}$ has less than or equal to $\ell'$ discontinuities on $[0,t]$. Hence,
\begin{align*}
\P(&(\cX^{(n,j)}_t)_{t\in[0,2\tau_n]}=(X^{(n,j)}_t)_{t\in[0,2\tau_n]}\text{ for } j=1,\dots,m_n)\\
&\ge 1-  \E\bigl[1\wedge  (\underbrace {m_n P^o_G((X_t)_{t\in[0,\tau_n]}\text{ has more than $\ell'$ discontinuities})}_{\to 0, \text{ in probability}})\bigr]\to 1.
\end{align*}

 
We call $j\in\{1,\dots,m_n\}$ an \emph{idealised}  $v$-indicator or $v$-escape, if we replace in the above definition the paths $(\cX_{t}^{(n,1)})_{t\in[0,2\tau_n]}, \dots, (\cX_{t}^{(n,m_n)})_{t\in[0,2\tau_n]}$ by the ``idealised'' paths $(X_{t}^{(n,1)})_{t\in[0,2\tau_n]}, \dots, (X_t^{(n,m_n)})_{t\in[0,2\tau_n]}$. We denote by $e^{(n)}_v$ the respective relative number of $v$-escapes when replacing the true exploration paths by the idealised ones. Note that, with high probability, the original and the idealised indicators/escapes coincide and $\mathfrak e_v^{(n)}$ agrees with $e_v^{(n)}$ for all $v$.

Obviously,  given $G^o$, each $j\in\{1,\dots,m_n\}$ is an idealised $v$-indicator independently of each other with probability 
$$P_G^o\bigl((X_t)_{t\in[0,\tau_n]} \text{ hits  }v\bigr)\to P_G^o\bigl((X_t)_{t\ge0}\text{ hits } v\bigr)>0$$
as $n\to\infty$.
Moreover, conditionally that a $j$ is a $v$-indicators it is additionally a $v$-escape independently with probability 
$$ P^v_G\bigl((X_t)_{t\in[0,\tau_n]}\text{ does not hit }V^o[R]\bigr)\to P^v_G\bigl((X_t)_{t\ge0}\text{ does not hit }V^o[R]\bigr).$$
Consequently, since $\partial V^o[R+1]$ is finite we get with the law of large numbers that 
$$
\sum_{v\in \partial V^{o}[R+1]} \bigl|e^{(n)}_v -P^v_G((X_t)\text{ does not hit }V^o[R])\bigr|\to 0, \text{ in probability.}
$$
and hence also
$$
\sum_{v\in \partial V^{o}[R+1]} \bigl|\mathfrak e^{(n)}_v -P^v_G((X_t)\text{ does not hit }V^o[R])\bigr|\to 0, \text{ in probability.}
$$ 

Next, we continue with a similar analysis for the paths obtained from the $K_n$-exploration of $G_n^{o_n}$. As before we introduce for every $n\in\N$,  \cadlag\ processes $(\bar X_t^{(n,1)})_{t\ge0}, (\bar X_t^{(n,2)})_{t\ge0},\dots$ such that given $G_n^{o_n}$ all processes are independent $P_{G_n}^{o_n}$-distributed and satisfy for $j\in\N$ and $t\ge0$, $\bar X_t^{(n,j)}=\bar \cX_t^{(n,j)}$ iff $\bar \cX^{(n,j)}\not=\partial$ (which is equivalent to $\bar X^{(n,j)}$ having  less than or equal to $\ell'_n$ discontinuities on $[0,t]$).
 In complete analogy we define the concept of $v$-indicators and $v$-escapes and their idealised versions for the explorations of $G_n^{o_n}$. We denote the relative number of escapes for $v$ by $\bar {\mathfrak e}_v^{(n)}$ and for their idealised versions by $\bar { e}_v^{(n)}$. If the $n$-th coupling succeeds and if the $R'$-neighbourhoods are disclosed by the exporations we have that $j$ is an $v$-indicator in the $G_n^{o_n}$-exploration if this is the case for $\varphi_n(v)$ in the $G^o$-exploration. Therefore, we have with high probability that
 $$
 \min_{v\in \partial V^{o_n}_n[R+1]} \# \{\text{idealised $v$-indicators in $G_n^{o_n}$-exploration}\} \to \infty,\text{ in probability.}
 $$
Moreover, note that given $G_n^{o_n}$ and that  $j$ is an idealised $v$-indicator it is also an idealised  $v$-escape with conditional probability
$$
P_{G_n}^{v} ((X_t)_{t\in[0,\tau_n]}\text{ does not hit $V^{o_n}[R]$}).
$$
Consequently,  by the law of large numbers as $n\to\infty$ 
$$
\sum_{v\in\partial V_n^{o_n}[R+1]}\bigl|\bar e_v^{(n)}- P_{G_n}^{v} ((X_t)_{t\in[0,\tau_n]}\text{ does not hit $V^{o_n}[R]$})\bigr|\to 0, \text{ in probability}.
$$
We recall that, with high probability, for all $v\in \partial V_n^{o_n}[R+1]$, $\bar e_v^{(n)}=\bar{\mathfrak e}_v^{(n)}=\mathfrak e_{\varphi_n(v)}^{(n)}$, so that, in probability,
\begin{align*}
\lim_{n\to\infty} \sum_{v\in\partial V_n^{o_n}[R+1]}\bigl| P_{G_n}^{v} &((X_t)_{t\in[0,\tau_n]}\text{ does not hit $V^{o_n}[R]$})\\
&- P_G^{\varphi_n(v)} ((X_t)_{t\ge0} \text{ does not hit } V^o[R] )\bigr|= 0.
\end{align*}
This entails the result.
\end{proof}

\section{Vacant set percolation}

Vacant set percolation has been considered on random graphs in various articles  \cite{CernyTeixeiraWindisch2010, CernyTeixeira2011,  CerHay20}. In this section, we relate the visiting measure to vacant set percolation similarly as done in \cite{Hofstad21} for the existence of giant components. 
\smallskip

We will  apply exploration convergence for doubly rooted graphs, i.e., graphs with  two designated vertices  $o$ and $o'$. That means that an exploration rule $K$ now acts on $\bar \IG^{\mathrm{finite,root}^2}$ being the set of finite, doubly-rooted, weighted, generalised graphs. Graph isomorphisms and exploration convergence is defined  in complete analogy, where the graph limit now is a random doubly-rooted graph $G^{o,o'}$, where each vertex is connected to at least one of the two roots.

We let for a doubly rooted graph $G^{o,o'}\in\IG^{\mathrm{root}^2}$, $V^{o,o'}[R]$ denote the set of vertices of $G$ that have distance less or equal to $R$ to one of the roots. Theorem~\ref{prop:2746} holds analogously for doubly rooted graphs.
\medskip

\begin{theorem}\label{thm_4_1}
Let $(\ell_n')_{n\in\N}$ and $(\ell_n)_{n\in\N}$ be $\N$-valued and $(\tau_n)_{n\in\N}$ and $(a_n)_{n\in\N}$ be $(0,\infty)$-valued with all three sequences tending to infinity and suppose that the random finite connected  $\IG^{\mathrm{root}^2}$-graphs  $G_1^{o_1,o_1'},\dots$  converge $(\ell_n)$-locally to the random $\IG^{\mathrm{root}^2}$-graph $G^{o,o'}$. Moreover, suppose that $\ell_n'=o(\ell_n)$, $\tau_n=o(\alpha_n(V_n))$, 
$$
\lim_{n\to\infty} d_{G_n}(\tau_n)=0, \ \ \lim_{n\to\infty}\max _{\bar o\in \{o,o'\}} P_G^{\bar o}\bigl(\# \text{ range}\bigl((X_t)_{[0,2\tau_n]}\bigr) >\ell_n'\bigr)   =0\text{ \ and }
$$
$$
\lim_{n\to\infty}\frac{\alpha_n(V_n)}{a_n}=\mathfrak a ,\text{ in probability.}
$$
Let $R,R'\in\N$ and $\sigma>0$. For $n\in\N$, let $\Xi_n$ be the $(V_n^{o_n,o_n'}[R],\tau_n,a_n)$-visiting measure and let $\Gamma$ be the Cox-process that is conditionally on $G^o$ a Poisson point process with intensity
\begin{align}\label{eq94624}
\sfrac 1{\mathfrak a} \mathrm{Leb}|_{[0,\infty)} \otimes P^{e_{V^{o,o'}[R],G}}_G
\end{align}
and let $\Gamma_n$ be the point process obtained from $\Gamma$ by killing all paths at time $\tau_n$.
 One can couple $(G^{o_n,o_n'}_n,\Xi_n)$ and $(G^{o,o'},\Gamma)$ such that, the following holds, with high probability:
\begin{itemize}
\item There are subsets $\bar{\mathcal V}_{n}\supset V_n^{o_n,o_n'}[R']$ and ${\mathcal V}_{n}\supset V^{o,o'}[R']$ of $V_n$ and $V$, respectively, and a graph isomorphism taking the $\bar {\mathcal V}_n$-disclosed graph of $G_n^{o_n,o_n'}$ to the $\mathcal V_n$-disclosed graph of $G^{o,o'}$.
\item The trajectories appearing in $\Xi_n|_{[0,\sigma]\times W_+}$ do not leave $\bar {\mathcal V}_n$ and the point process $\Gamma_n|_{[0,\sigma]\times W_+}$ is obtained from  $\Xi_n|_{[0,\sigma]\times W_+}$ by applying the isomorphism  $\varphi_n$ on the states of the trajectories. 
\end{itemize}
 \end{theorem}
 
 \begin{proof}
 Literally the same proof where one now uses path explorations from the two roots with alternating starting points.
 \end{proof}

Let us now introduce the setting of vacant set percolation. In the following, we consider site- and bond-percolation. Both can be treated in complete analogy. Let $(G_n)$ be a sequence of  random finite, connected $\IG$-graphs  and let $(a_n)$ be a sequence of  strictly positive numbers that is used to govern the time scale (as in the definition of the visiting counting measure). For $\sigma\ge 0$ we denote by $G_n(\sigma)$ the graph that is obtained when removing the sites, resp.\ bonds seen by the Markov chain until time $\sigma a_n$. For a vertex $v\in V_n$ we let $\cC^v_n(\sigma)$ the component of $v$ in $G_n(\sigma)$.

Again we assume that the uniformly rooted graphs converge (in the sense of exploration convergence)  to some $\IG$-graph $G^o$. We denote by $G^o(\sigma)$ the graph that is obtained by applying interlacement percolation (site or bond percolation). More explicitly, we use Remark~\ref{rem:235462} and consider a doubly stochastic Poisson process that is, given the rooted graph $G^o$, a Poisson point process with intensity measure $\mathrm{Leb}|_{[0,\infty)}\otimes \nu_G$. Removing all sites, resp.\ edges, seen by the path with time index less than or equal to $\sigma$ we obtain $G(\sigma)$. 
We relate the relative size of the largest component in $\cC_n(\sigma)$ with the probability of $o$ being in an infinite component of $G(\sigma)$. 
\medskip

\begin{theorem}\label{thm_4_2}Let $\sigma>0$ and let $G_1,G_2,\dots$ random finite, connected $\IG$-graphs. 
Suppose that the assumptions of the latter theorem are satisfied when choosing independent uniform roots $o_n$ and $o_n'$ from the graph $G_n$ and that the graph limit is in distribution equal to  two independent unconnected copies of $G^o$ (with $G^o$ itself being the graph limit of $G^{o_n}_n$). Then the following is equivalent for every $\sigma\ge0$:
\begin{enumerate}\item One has 
$$
\lim_{k\to\infty} \limsup_{n\to\infty} \P\bigl(\#\cC_n^{o_n}(\sigma)\ge k,\#\cC_n^{o_n'}(\sigma)\ge k, o_n \not\in \cC_n^{o_n'}(\sigma)\bigr)=0.
$$
\item One has, in probability, that
$$
\lim_{n\to\infty} \frac 1{\#V_n} \#\cC_n^\mathrm{max}(\sigma)= \P(\#\cC^o(\sigma)=\infty).
$$
\end{enumerate}
\end{theorem}

\smallskip

\begin{lemma}
For $n\in\N$ and nonnegative reals $a_1,\dots,a_n$ one has
$$ 
(a_1+\ldots+a_n)^2-3\sum_{i<j} a_i a_j\le \max(a_1,\dots,a_n)^2\le(a_1+\ldots+a_n)^2-2\sum_{i<j} a_i a_j .
$$
\end{lemma}

\begin{proof}1) We start with proving the first inequality. The proof is established by induction over $n\in\N$. Obviously the statement is true for $n=1$. 
Now suppose that the statement is true for $n\in\N$. We let $a_1,\dots,a_{n+1}$ with the maximum being attained in the first $n$ entries (otherwise we rearrange the summands). Then
\begin{align*}
&(a_1+\ldots+a_n+a_{n+1})^2-3\sum_{1\le i<j\le n+1}a_i a_j \\
&\quad =\underbrace{(a_1+\dots+a_n)^2 - 3 \sum_{1\le i<j\le n}a_i a_j}_{\le \max(a_1,\dots,a_n)^2}  +\underbrace{ a_{n+1}^2- (a_1+\ldots+a_n) a_{n+1}}_{\le 0}\\
&\quad \le \max(a_1,\dots,a_{n+1})^2.
\end{align*}
2)  It remains to prove the second inequality. One has
$$
(a_1+\ldots+a_n)^2= \sum_{i=1}^n a_i^2+2\sum_{i<j} a_i a_j\ge \max(a_1,\dots,a_n)^2+2\sum_{i<j} a_i a_j.
$$
\end{proof}

\begin{proof}[Proof of Theorem~\ref{thm_4_2}]
First we analyse the random variable $N_n^{(k)}=\frac 1{\#V_n} \sum_{v\in V_n} \1_{\{\#\cC_n^v(\sigma)\ge k\}}$ that counts the relative number of vertices in components with at least $k$ entries after vacant set percolation. One has as consequence of Thm.~\ref{thm_4_1} that
\begin{align*}
\E\bigl[ \bigl(N_n^{(k)}\bigr)^2\bigr]&=\P(\#\cC^{o_n}_n(\sigma)\ge k, \#\cC^{o_n'}_n(\sigma)\ge k)\\
&\to \P(\#\cC^{o}(\sigma)\ge k,\#\cC^{o'}(\sigma)\ge k)=\P(\#\cC^{o}(\sigma)\ge k)^2,
\end{align*}
where we used that the event  $\{\#\cC^{o}(\sigma)\ge k,\#\cC^{o'}(\sigma)\ge k\}$ is decidable on the basis of the visiting measure.
Analogously,
$$
\E\bigl[ N_n^{(k)}\bigr]=\P(\#\cC^{o_n}_n(\sigma)\ge k)\to \P(\#\cC^{o}(\sigma)\ge k).
$$
Consequently, $N_n^{(k)}$ converges, as $n\to\infty$, in probability, to $\P(\#\cC^{o}(\sigma)\ge k)$ which itself converges as $k\to\infty$ to $\P(\#\cC^{o}(\sigma)=\infty)$.

1) We first  prove that (a) implies (b).

1.a) \emph{Lower bound.}
For fixed $n\in\N$ we apply the lemma with $a_1,\dots,a_m$ being the sizes of the components that contain more than $k$ entries. Then we get
\begin{align}\label{eq7247}
\bigl(\underbrace{\sum_{v\in V_n} \1_{\{\#\cC^v_n(\sigma)\ge k\}}}_{=\# V_n\, N_n^{(k)}}\bigr)^2- \frac32 \sum_{v,w\in V_n}   \1_{\{\#\cC^v_n(\sigma)\ge k, \#\cC^w_n(\sigma)\ge k, v\not\in \cC^w_n(\sigma)\}} \le \bigl(\#\cC^\mathrm{max}_n(\sigma)\bigr)^2.
\end{align}
Now for fixed $\epsilon>0$ we can pick $k\in\N$ sufficiently large such that for large $n\in\N$
\begin{align*}
\E\Bigl[&\frac 1{(\#V_n)^2}\sum_{v,w\in V_n}   \1_{\{\#\cC^v_n(\sigma)\ge k, \#\cC^w_n(\sigma)\ge k, v\not\in \cC^w_n(\sigma)\}}\Bigr]\\
&= \P(\#\cC^{o_n}_n(\sigma)\ge k, \#\cC^{o_n'}_n(\sigma)\ge k, o_n\not\in \cC^{o_n'}_n(\sigma))\le \epsilon.
\end{align*}
Recalling that $N_n^{(k)}\to \P(\#\cC^o(\sigma)\ge k)$, in probability, we can apply  a diagonalisation argument and choose $k(n)$ in dependence on $n\in\N$ such that $k(n)\to\infty$, $N_n^{(k(n))}\to \P(\#\cC^{o}(\sigma)=\infty)$, in probability, and 
$$
\E\Bigl[\frac 1{(\#V_n)^2}\sum_{v,w=1}^n   \1_{\{\#\cC^v_n(\sigma)\ge k(n), \#\cC^w_n(\sigma)\ge k(n), v\not\in \cC^w_n(\sigma)\}}\Bigr]\to 0.
$$
With~(\ref{eq7247}) we get that
$$
\frac 1{\#V_n} \#\cC_n^\mathrm{max}(\sigma)\ge \P(\#\cC^{o}(\sigma)=\infty)-o_P(1).
$$

1.b) \emph{Upper bound}. Note that $\cC_n^\mathrm{max} (\sigma)\le k\vee \sum_{v\in V_n}\1_{\{\#\cC_n^v(\sigma)\ge k\}}$. Note that for every $k\in\N$,
 $N_n^{(k)}\to \P(\#\cC^o(\sigma)\ge k)$ and $\# V_n\to \infty$, in probability. Again a diagonalisation argument yields a sequence $(k(n))$ tending to infinity so that $N_n^{(k(n))}\to \P(\#\cC^o(\sigma)=\infty)$ and $k(n)/\#V_n\to 0$, in probability. Consequently,
$$ \frac 1{\# V_n} \cC_n^\mathrm{max} (\sigma)\le  \P(\#\cC^{o}(\sigma)=\infty)+o_P(1).
$$

2) It remains to show that (b) implies (a).
As above choosing $a_1,\dots,a_m$ as the sizes of the random components after vacant set percolation with more than $k$ vertices we get that whenever $\#\cC^\mathrm{max}_n(\sigma)\ge k$ or, equivalently, $N_n^{(k)}>0$, one has
$$\Bigl(\frac 1{\#V_n}\#\cC^\mathrm{max}_n(\sigma)\Bigr)^2\le (N_n^{(k)})^2- \underbrace{ \frac 1{(\# V_n)^2} \sum_{v,w\in V_n} \1_{\{\#\cC^{o_n}_n(\sigma)\ge k, \#\cC^{o_n'}_n(\sigma)\ge k, o_n\not\in \cC^{o_n'}_n(\sigma)\}}}_{=:\Delta_n^{(k)}}.$$
We recall that $N_n^{(k)}\to \P(\#\cC^o(\sigma)\ge k)$, in probability.
Now suppose that (a) is not satisfied. Then there is $\epsilon>0$ such that along a subsequence of $n$'s one has for all $k\in\N$ that for large $n$
$$
\P(\#\cC_n^{o_n}(\sigma)\ge k,\#\cC_n^{o'_n}(\sigma)\ge k, o_n\not\in \cC_n^{o_n'}(\sigma))\ge \epsilon.
$$
 Without loss of generality we can assume that the latter is  true along the whole sequence since we otherwise may drop the graphs that are not in the subsequence.
 In particular, we get that as $n\to\infty$
 $$
 \epsilon\le \P(\#\cC_n^{o_n}(\sigma)\ge k) \to \P(\#\cC ^o(\sigma)\ge k)
 $$
 so that $ \P(\#\cC ^o(\sigma)=\infty)=\lim_{k\to\infty} \P(\#\cC ^o(\sigma)\ge k)\ge \epsilon$.
  We can now choose an increasing  $(k(n))_{n\in\N}$ tending to infinity such that
 $$ \frac 1{\# V_n} k(n)\to0 \ \text{  and } \ N_n^{(k(n))}\to \P(\#\cC^o(\sigma)=\infty)\text{, \  in probability,}$$
  and,  for all large $n\in\N$,
$$\P(\#\cC_n^{o_n}(\sigma)\ge k(n),\#\cC_n^{o'_n}(\sigma)\ge k(n), o_n\not\in \cC_n^{o_n'}(\sigma))\ge \epsilon.$$ Then using that $\Delta_n^{(k(n))}$ is bounded by one we conclude that
$$
\epsilon \le \P(\#\cC_n^{o_n}(\sigma)\ge k(n),\#\cC_n^{o'_n}(\sigma)\ge k(n), o_n\not\in \cC_n^{o_n'}(\sigma))= \E[\Delta_n^{(k(n))}]\le \frac \epsilon 2 + \P\bigl(\Delta_n^{(k(n))}\ge \frac \epsilon 2\bigr)
$$
for large $n$ and thus recalling that $N_n^{k(n)}\to\P(\#\cC^o(\sigma)=\infty)\ge \epsilon$ we have with high probability that
$$
\Bigl(\frac 1{\#V_n}\#\cC^\mathrm{max}_n(\sigma)\Bigr)^2\le   (N_n^{(k(n))})^2-\Delta_n^{(k(n))}=\P(\#\cC^{o}(\sigma)=\infty)^2+o_P(1)-\Delta_n^{(k(n))}.
$$
Since $\P(\Delta_n^{(k(n))}\ge \frac \epsilon 2)\ge \frac\epsilon2$ we get that $\frac 1{\#V_n}\#\cC^\mathrm{max}_n(\sigma)$ does not converge to $\P(\#\cC^{o}(\sigma)=\infty)$, in probability.
\end{proof}
\black

\begin{appendix}
\section{Exploration convergence for the configuration model}



For every $n\in\N$ we denote by $\mathbf{d}^{(n)} =(d_k^{(n)})_{k=1,\dots,n}$ an $\N_0$-valued sequence so that $N^{(n)}:=\sum_{k=1,\dots,n} d_k^{(n)}\in 2\N$. We denote by $D_n$ an $\N_0$-valued random variable with distribution
$$
\frac 1n\sum_{k=1}^n \delta_{d_k^{(n)}}
$$ 
and let $D_n^*$ denote a $\N$-valued random variable with size biased distribution distribution
$$
\frac 1{N^{(n)}}\sum_{k=1}^n d_k^{(n)} \delta_{d_k^{(n)}}
$$
so that
$$
\P(D_n^*=k)=\frac1{\E[D_n]}\,k\, \P(D_n=k).
$$
The random graph with fixed degree sequence $(\mathbf d^{(n)})_{n\in\N}$ is a sequence of random multigraphs that is formed by pairing the even number of half-edges in 
$$
\IH_n=\bigcup_{j\in\{1,\dots,n\}}  \{j\} \times \{1,\dots, d_j^{(n)}\}
$$ 
and forming the random multigraph  $G_n$ by taking as vertex set $\{1,\dots,n\}$ and establishing for each $\langle (j_1,r_1),(j_2,r_2)\rangle$ of the latter random pairing an edge connecting $j_1$ and $j_2$. A closely related model is the configuration model, where the degrees itself are chosen randomly and then the graph is formed by choosing a random fixed degree graph for this configuration as above.

As is well known fixed degree sequence graphs converge locally to  Galton-Watson trees with a modified offspring distribution for the root vertex.
For a $\N_0$-valued random variable $D$ and a $\N$-valued random variable $D^*$ we call a random rooted graph~$G^o$ a $\mathrm {GWP}(D,D^*)$-tree, if the root has a $D$-distributed number of descendants and if every descendant has itself an independent  $D^*-1$-distributed number of descendants.
\medskip

\begin{theorem}[Coupling explorations for the random fixed degree graph] \label{theo:12} Let $(\mathbf d^{(n)})_{n\in\N}$ be a sequence as above satisfying $\liminf_{n\to\infty}\,\E[D_n]>0$.
Moreover, let $(\ell_n)_{n\in\N}$ be $\N$-valued with
\begin{align}\label{eq87346}
\ell_n=o\bigl(\sqrt n \, \E[D_n^*\wedge \sqrt n]^{-1}\bigr).
\end{align}

Let $(\bar G_n)_{n\in\N}$ denote a sequence of fixed degree random graphs with degree sequence $(\mathbf d^{(n)})_{n\in\N}$ as introduced above and choose independent on $\{1,\dots,n\}$ uniformly distributed roots $\bar o_n$. Moreover, let for each $n\in\N$, $G_n^{o_n}$ be a $\mathrm {GWP}(D_n,D_n^*)$-tree. 
Then for every sequence $(K_n)$  of exploration rules, one can couple the $\ell_n$-step $K_n$-explorations $\bar{\mathfrak E}_{\ell_n}^{(n)}$ and ${\mathfrak E}_{\ell_n}^{(n)}$ of $\bar G_n^{\bar o_n}$ and $G_n^{o_n}$  such that they agree with high probability as $n\to\infty$.
\end{theorem}
\smallskip

Before we prove the theorem, we discuss under which conditions the $\mathrm{GWP}(D_n,D_n^*)$-trees $G_n^{o_n}$ can be replaced by a $\mathrm{GWP}(D,D^*)$-tree, where $D$ and $D^*$ are weak limits of $D_n$ and $D_n^*$. Moreover, we will discuss the constraint~(\ref{eq87346}) in the case of polynomially decaying tail probabilities of $\P_{D_n}$.
\begin{lemma}\label{le:8246}
For two $\N_0$-valued random variables $D_1$ and $D_2$, two $\N$-valued random variables $D_1^*$ and $D_2^*$, $\ell\in\N$ and an exploration rule one can couple the $\ell$-step $K$-explorations of a $\mathrm{GWP}(D_1,D_1^*)$- and a $\mathrm{GWP}(D_2,D_2^*)$-tree such that they agree with probability greater or equal to
$$
1-d_\mathrm{TV}(\P_{D_1},\P_{D_2})-\ell \,d_\mathrm{TV}(\P_{D_1^*},\P_{D_2^*}).
$$
\end{lemma}

\begin{proof}
The initial explorations can be coupled such that they agree with probability
$1-d_\mathrm{TV}(\P_{D_1},\P_{D_2})$. Moreover, in each exploration step in which  a not yet disclosed vertex is explored we can apply a coupling hat succeeds with conditional probability $1-d_\mathrm{TV}(\P_{D_1}^*,\P_{D_2}^*)$.
\end{proof}

\begin{remark}\label{rem:246}\begin{enumerate}\item  If in Theorem~\ref{theo:12} additionally for random variables $D$ and $D^*$ one has 
\begin{align}\label{eq84621}
d_\mathrm{TV}(\P_{D_n},\P_{D})+\ell_n \,d_\mathrm{TV}(\P_{D_n^*},\P_{D^*})\to 0,
\end{align}
then the result remains true, when letting $G_n^{o_n}$ be $\mathrm{GWP}(D,D^*)$-trees instead (consequence of Remark~\ref{le:8246}).
\item Let $(p_k)_{k\in\N_0}$ be a sequence of probability weights with $p_0\not=1$. Then one may choose a degree sequence $(\mathbf d^{(n)})_{n\in\N}$ such that $|\P(D_n=k)-p_k|\le \frac 1n\wedge p_k$. We let
$\epsilon_n=\sum_{k=1}^\infty (\frac 1n\wedge p_k)$. Then
\begin{align*}
d_\mathrm{TV}&(\P_{D_n^*},\P_{D^*})= \sum_{k=1}^\infty k \Bigl| \frac {\P(D_n=k)}{\E[D_n]}-\frac {p_k}{\E[D]}\Bigr| \\
&\le \frac 1{\E[D]} \sum_{k=1}^\infty k |\P(D_n=k)-p_k|+\frac {|\E[D]-\E[D_n]|}{\E[D]}\le 2\frac{\epsilon_n}{\E[D]}.
\end{align*}
Thus for given distribution $(p_k)_{k\in\N_0}$ we can choose a  degree sequence  $(\mathbf d^{(n)})_{n\in\N}$ as above and can replace as described in (a) the GWP-tree if  $\ell_n \epsilon_n\to0$. An elementary computation shows that if $p_k=\mathcal O(k^{-\tau})$ for a $\tau>2$, then $\epsilon_n=\mathcal O(n^{-1+\frac 2\tau})$ and one can choose $(\ell_n)$ of order $o(n^{1-\frac2\tau})$ and has validity of~(\ref{eq84621}).
\item We analyse constraint~(\ref{eq87346}) on $\ell_n$. Suppose that we are given a degree sequence $(\mathbf d^{(n)})_{n\in\N}$ such that for a $\tau>2$ and finite $C$ one has $\P(D_n=k) \le  C k^{-\tau}$ for $k\in\N$ and $n\in\N$. Moreover assume that $\liminf_{n\to\infty} \E[D_n]>0$.  
Then
$$
\E[D_n^* \wedge \sqrt n] =\frac 1{\E[D_n]} \sum_{k=1}^\infty k\, \P(D_n=k) (k\wedge \sqrt n)\le \frac C{\E[D_n]} \sum_{k=1}^\infty k^{1-\tau}\, (k\wedge \sqrt n)
$$
Elementary calculus gives that
$$
\sum_{k=1}^\infty  k^{1-\tau}  (k\wedge \sqrt n)\approx \begin{cases} n^{(3-\tau)/2},& \text{ if }\tau \in(2,3),\\
 \log n,& \text{ if }\tau =3,\\
  1,& \text{ if }\tau >3.\\
\end{cases}
$$
Hence, (\ref{eq87346}) is satisfied, if
\begin{align}\label{eq8426}
\ell_n = \begin{cases}
o(n^{\frac \tau2-1}), &\text{ if } \tau<3,\\
o(\sqrt n(\log n)^{-1}), &\text{ if } \tau=3,\\
o(\sqrt n), & \text{ if } \tau>3.
\end{cases}
\end{align}
\item When working with degree sequences $(\mathbf d^{(n)})_{n\in\N}$ as in (b), one also apply (c) and gets that  Theorem~\ref{theo:12} is applicable with  $G_n^{o_n}$ being replaced by $\mathrm{GWP}(D,D^*)$-trees if $(\ell_n)$ is of order $o\bigl(n^{(1-\frac 2\tau)\wedge \frac 12}\bigr)$. Indeed, for $\tau\in(2,3]$, the constraint of (c) is less restrictive than the one of (b) and for $\tau\ge4$, we get as assumption that $(\ell_n)$ is of order $o(\sqrt n)$.
\end{enumerate}
\end{remark}

\begin{proof}[Proof of Theorem~\ref{theo:12}]
In the proof we will stop explorations once we find vertices with exceptionally large degree: we let $a_n=\lfloor\sqrt n\rfloor$ for $n\in\N$ and note that 
\begin{align}\label{eq83213}
\ell_n= o\bigl( \P(D_n^*\ge a_n)^{-1}\bigr) \ \text{ and } \   \ell_n =o\Bigl( \E[\1_{\{D_n^* <a_n\}} D_n^*]^{-1} \sqrt n\Bigr)
\end{align}
as consequence of~(\ref{eq87346}).

We fix a sequence of exploration rules $(K_n)$.
We directly introduce a coupling of two explorations $\bar{ \mathfrak E}_{\ell_n}$ and $\mathfrak E_{\ell_n}$ of $\bar G_n^{\bar o_n}$ and $G_n^{o_n}$, respectively. Formally, $\bar{ \mathfrak E}_{\ell_n}$ will not be a $K_n$-exploration of $\bar G_n^{\bar o_n}$. However, it will be very close in an appropriate sense to be discussed later.

We  start with the construction of $\mathfrak E^{(n)}_{\ell_n}$. In particular, this will include a definition of the graph $G_n^{o_n}$.
First we pick a label $o_n$ from $[n]$ uniformly at random and form $ \mathfrak E_{\ell_n}$ as the rooted graph with edges labeled $(o_n,1),\dots, (o_n,d_{o_n}^{(n)})\in \IH_n$ that connect $o_n$ with  $\partial$ with no edge history.
Then we iteratively do the following for $k=1,2,\dots$:
\begin{enumerate}
\item We pick a $K_n(\mathfrak E^{(n)}_{k-1},\,\cdot\,)$-distributed edge $h_k'$ of the graph of the exploration $\mathfrak E^{(n)}_{k-1}$ and add $h_k'$ to the edge history.
\item If $h_k'$ points in $\mathfrak E^{(n)}_{k-1}$ to $\partial$,  we choose an on $\IH_n$ uniformly distributed half edge $h_k''$  and  
 build the graph of $\mathfrak E^{(n)}_k$ by
\begin{enumerate}
\item  by redirecting $h_k'$ to a newly added vertex being labelled by the vertex $v_k$ being associated to $h_k''$ (if the label has not appeared yet) 
\item which has $\mathrm {deg}_{n}(v_k)-1$ further edges pointing towards $\partial$ with labels $\IH_n(v_k)\backslash \{h_k''\}$ (if the labels have not appeared yet). 
\end{enumerate}
\end{enumerate}
Obviously the result is an $K_n$ exploration of a random rooted tree $G_n^{o_n}$ as described above. 

We couple this exploration with an ``exploration'' of $\bar G_n^{\bar o_n}$ as follows. We start with the vertex $\bar o_n=o_n$ and attach to it edges that are indexed by  $\IH_n(o_n)$ all connecting $o_n$ and $\partial$ so that $\bar{ \mathfrak E}^{(n)}_0=\mathfrak E^{(n)}_0$. Then we iteratively do the following for $k=1,2,\dots$:
\begin{enumerate}
\item If  $\bar{ \mathfrak E} ^{(n)}_{k-1}$ agrees with $\mathfrak E^{(n)}_{k-1}$, then we choose $\bar h_k'=h_k'$  and otherwise we choose an independent $K_n( \bar{ \mathfrak E} ^{(n)}_{k-1},\,\cdot\,)$-distributed $\bar h_k'$ and we continue as follows if $\bar h_k'\not=\partial$.
\item If   $h_k''\not\in \{h_1',\dots,h_k',h_1'',\dots,h_{k-1}''\}$ we set $\bar h_k''=h_k''$ and otherwise we draw $\bar h_k''$ uniformly at random from $\IH_n\backslash  \{\bar h_1',\dots,\bar h_k',\bar h_1'',\dots,\bar h_{k-1}''\}$.
\item We distinguish two cases: 
\begin{enumerate}\item if there is no $\bar h_k''$-indexed edge present in $\bar {\mathfrak E}_{k-1}^{(n)}$, then we reconnect the edge $\bar h_k'$ towards a newly formed vertex with the to $\bar h_k''$ associated label $\bar v_k$ (which equals $v_k$, if we have $\bar h_k''=h_k''$) and attach to it further edges indexed by $\IH_n(\bar v_k)\backslash \{\bar h_k''\}$
 \item if there is an $\bar h_k''$-indexed edge already present  in the exploration $\mathfrak E^{(n)}_{k-1}$, then it has to be connected to $\partial$ (since it was not yet connected in previous steps) and  we connect the two involved vertices by one edge and use as edge one of the two identifiers uniformly at random (the other edge is removed).
\end{enumerate}
\end{enumerate}

First note that $\mathfrak E_{\ell_n}^{(n)}$ coincides with $\bar {\mathfrak E}_{\ell_n}^{(n)}$ if  in the first $\ell_n$ steps in (b) one always chooses $\bar h_k''=h_k''$ and in (c) always case (i) enters. 
We  let $\cN_k^{(n)}$ denote the number of half-edges in $\mathfrak E_k^{(n)}$ (where an edge connecting two vertices from $\II$ is counted twice and an edge leading to $\partial$ only once).
For $k=1,\dots,\ell_n$, we denote 
$$
\cF_k=\bigl\{\mathfrak E_k^{(n)}\not=\bar{\mathfrak E}_k^{(n)},\, \mathfrak E_{k-1}^{(n)}=\bar{\mathfrak E}_{k-1}^{(n)}, \,\mathrm{degmax} (\mathfrak E_{k-1}^{(n)})<a_n\bigr\},
$$
where $\mathrm{degmax} (\mathfrak E_{k-1}^{(n)})$ denotes the maximal degree appearing in the graph of the exploration $\mathfrak E_{k-1}^{(n)}$.
Moreover, we denote $\cF_0=\{\mathfrak E_0^{(n)}=\bar {\mathfrak E}_0^{(n)}\}$ which is the sure event.
Then
\begin{align*}
\P(\cF_k)&\le \frac 1{N^{(n)}} \E\bigl[ \1_{\{\mathrm{degmax} (\mathfrak E_{k-1}^{(n)}<a_n)\}}
\cN_{k-1}\Bigr]\\
& \le \frac1{N^{(n)}}\bigl(\E[\1_{\{D_n<a_n\}} D_n]+(k-1) \,\E[\1_{\{D_n^*<a_n\}} D_n^*]\bigr)
\end{align*}
 and we get that the probability for the coupling to fail within the first $\ell_n$ steps is bounded by
 \begin{align*}
\P\bigl(&\mathrm{degmax} (\mathfrak E_{\ell_n}^{(n)})\ge a_n\bigr)+ \sum_{k=1}^{\ell_n}\P(\cF_k)\\
& \le \underbrace{\P(D_n\ge a_n)}_{\le \P(D_n^*\ge a_n)} +\ell_n\, \P(D_n^*\ge a_n) + \frac 1{N^{(n)}} \sum_{k=1}^{\ell_n}\bigl(\E[D_n]+(k-1) \E[\1_{\{D_n^*<a_n\}}D_n^*]\bigr)\\
&\le (\ell_n+1)\,\P(D_n^*\ge a_n)+  \frac 1{2 N^{(n)}} \ell_n^2\, \E[\1_{\{D_n^*<a_n\}} D_n^*]+\frac {\ell_n}n.
 \end{align*}
By (\ref{eq83213}),   $\ell_n^2 \, \E[\1_{\{D_n^*<a_n\}} D_n^*]^2= o(n)$ and since $\E[\1{\{D_n^*<a_n\}} D_n^*]\ge \P(D_n^*<a_n)\to1$ and  $N^{(n)}/n= \E[D_n]$ whose liminf is strictly bigger than zero, we get that $\ell_n^2\, \E[\1_{\{D_n^*<a_n\}} D_n^*]=o(N^{(n)})$. Since by (\ref{eq83213}) also $(\ell_n+1)\,\P(D_n^*\ge a_n)\to0$ we get that the $\ell_n$-step coupling succeeds with high probability.

Now note that $\bar{ \mathfrak E}_{k}^{(n)}$ is in the sense of our definition not a $K_n$-exploration  of any graph since the graphs appearing in $\bar{ \mathfrak E}_{1}^{(n)},\dots$ are not necessarily consistent in the sense that in step (c.ii) edges are removed and paired newly what can't happen when exploring a graph.
To fix this and rigorously relate $\bar{ \mathfrak E}_{\ell_n}^{(n)}$ to a random fixed degree graph we form $\bar G_n$ by pairing the half-edges in $\IH_n$ that have not been paired in the graph of $\bar{ \mathfrak E}_{\ell_n}^{(n)}$ uniformly at random (with ties as in (c.ii) being treated as above). We stress that now $\bar G_n^{\bar o_n}$ is indeed a random fixed degree graph for the degree sequence $\mathbf d^{(n)}$ with an independent uniformly chosen root vertex $o_n$.

Now we are in the position to define a $K_n$-exploration of $G_n^{o_n}$. We let $\bar{\bar{\mathfrak E}}_0^{(n)}$ be the $\{\bar o_n\}$-disclosed graph of $\bar G_n^{\bar o_n}$ together with an empty history of edges. Then for $k=1,\dots,\ell_n$, do the following
\begin{enumerate}
\item if $\bar{\bar{\mathfrak E}}_{k-1}^{(n)}$ agrees with ${\bar{\mathfrak E}}_{k-1}^{(n)}$, then form $\bar{\bar{\mathfrak E}}_{k}^{(n)}$ by adding $\bar h'_k$ to the history of explored edges and form the graph as the appropiately disclosed graph of $\bar G_n^{\bar o_n}$;
\item otherwise choose a $K_n( \bar{\bar{\mathfrak E}}_{k}^{(n)} ,\,\cdot\,)$-distributed edge $\bar{\bar h}_k'$ and  form $\bar{\bar{\mathfrak E}}_{k}^{(n)}$ by exploring $\bar{\bar h}_k'$ in $\bar G_n^{\bar o_n}$ instead.
\end{enumerate}  
We denote by  $\cG_{k}^{(n)}$ the graph $G_n$ restricted to the vertex set appearing in $\bar {\mathfrak E}_{k}^{(n)}$. We note that  $\cG_k^{(n)}$ does not contain circles if and only if
$\bar {\mathfrak E}_{k}^{(n)}=\bar {\bar{\mathfrak E}}_{k}^{(n)}$.
Now  given $\bar{\mathfrak E} _{k-1}^{(n)}$ the event that $\cG_{k-1}^{(n)}$ does not contain circles is independent of the choice of the vertex $\bar h_k'$. Therefore, $(\bar{\bar{\mathfrak E}}_{k}^{(n)})$ is a $K_n$-exploration of $G_n^{o_n}$. Moreover, $(\bar{\bar{\mathfrak E}}_{\ell_n}^{(n)})$ agrees with $({\bar{\mathfrak E}}_{\ell_n}^{(n)})$ if and only if $\cG^{(n)}_{\ell_n}$ does not contain cycles.

Note that in the case that ${\mathfrak E}_{\ell_n}^{(n)}=\bar{\mathfrak E}_{\ell_n}^{(n)}$ the graph $\cG_{\ell_n}^{(n)}$ contains circles, if there are edges in $\bar{\mathfrak E}_{\ell_n}^{(n)}$ that point to $\partial$ that are actually connected to each other in $\bar G_n$. On $\{\mathfrak E_{\ell_n}^{(n)}=\bar{\mathfrak E}_{\ell_n}^{(n)}\}$, given $\bar{\mathfrak E}_{\ell_n}^{(n)}$ the probability of finding no circle when matching the remaining edges is bounded from above by
\begin{align}\label{eq924762}
\frac {(\cN_{\ell_n}^{(n)})^2}{N^{(n)}-\cN^{(n)}_{\ell_n}}.
\end{align}
By construction of the $\mathfrak E^{(n)}$-exploration we may define (on a possibly enlarged probability space) a sequence of independent random variables  $\cD^{(n)}_0, \cD^{(n)}_1,\dots$ such that $\cD^{(n)}_0$ is a copy of $D_n$ and $\cD^{(n)}_1,\ldots$ are copies of $D_n^*$ and such that 
$\cN_{k}^{(n)}\le\cD^{(n)}_0+\ldots+\cD^{(n)}_{k}$ for $k=0,\dots,\ell_n$. 
We note that $(\cD^{(n)}_0:n\in\N)$ is tight. Moreover, we have with 
 high probability that  $\max(\cD^{(n)}_1,\ldots,\cD^{(n)}_{\ell_n})<a_n$. Together with
$$
\E\bigl[\1_{\{\max(\cD^{(n)}_1,\ldots,\cD^{(n)}_{\ell_n})<a_n\}}\cD_1^{(n)}+\ldots+\cD_{\ell_n}^{(n)}\bigr]\le \ell_n \,\E[\1_{\{D_n^*<a_n\}} D_n^*] =o(\sqrt n)
$$
and the  Markov inequality we conclude that $\cN^{(n)}_\ell=o_P(\sqrt n)$. Since $N^{(n)}$ is of order~$n$ we conclude that the probability in~(\ref{eq924762}) converges to zero, in probability.
Thus we showed that, with high probability, $\mathfrak E^{(n)}_{\ell_n}={\bar {\mathfrak E}}^{(n)}_{\ell_n}=\bar{\bar {\mathfrak E}}^{(n)}_{\ell_n}$.
\end{proof}
\medskip

\begin{theorem}[Local convergence for the configuration model]
Let  $(p_r)_{r\in\N_0}$ be a sequence of probability weights with $p_0\not=1$ and suppose that for a $\tau\in(2,\infty)$
$$
p_r =\cO(r^{-\tau})\text{ \ as \ }r\to\infty.
$$
Let $G_n$ be a configuration model with weights $(p_r)_{r\in\N_0}$ meaning that 
we take a sequence of i.i.d.\ $(p_r)$-distributed random variables $(q_k)_{k\in\N}$, let
$$
d_k^{(n)}=q_k+\1_{\{k=n\}} \1_{\{\sum_{v=1}^n q_v\text{ is odd}\}}
$$
and define $(G_n)$ in such a way that given $\mathbf d^{(n)}=(d^{(n)}_k)_{k=1,\dots,n}$, $G_n$ is a random graph with fixed degree sequence $\mathbf d^{(n)}$.
We suppose that $D$ and $D^*$ are random variables with weights $(p_r)_{r\in\N_0}$ and $(p^*_r)_{r\in\N_0}$ given by
$$
p^*_r= \frac 1{\E[D]} r p_r.
$$
If
$$
\ell_n=\begin{cases} o(n^{\frac{\tau- 2}\tau}), & \text{ if } \tau<4, \\ 
o(n^{\frac 12}(\log n)^{-1}), & \text{ if }\tau=4,\\
o(n^{\frac 12}), & \text{ if }\tau>4,\end{cases}
$$
then $(G_n^{o_n})$ with an independent uniformly chosen root converges $(\ell_n)$-locally to a $\mathrm{GWP}(D,D^*)$ random tree $G^o$.
\end{theorem}

\begin{proof}
We will apply Theorem~\ref{theo:12} and Remark~\ref{rem:246} (a). with the choice $a_n=n^{1/\tau}$.
We analyse the analogues of $\E[D_n]$, $\P(D_n^*\ge a_n)$, $\E[\1_{\{D_n^* < a_n\}} D_n^*]$ and $d_\mathrm{TV}(D_n^*,D^*)$ which are now random and which depend on the degree sequence $\mathbf d^{(n)}$. We need to consider
$$
\mathfrak m_1^{(n)}=\frac 1n \sum_{k=1}^n d_k^{(n)},\ \mathfrak p_n=\frac 1{n\mathfrak m_1^{(n)} }\sum_{k=1}^n \1_{\{d_k^{(n)}\ge a_n\}}  d_k^{(n)},\
 \mathfrak m_2^{(n)} = \frac 1{n\mathfrak m_1^{(n)} }\sum_{k=1}^n\1_{\{d_k^{(n)}<a_n\}}  (d_k^{(n)})^2
 $$
 and
 $$\mathfrak d^{(n)}=\ \frac12 \sum_{k=1}^\infty \Bigl| \frac 1{n \mathfrak m_1^{(n)}} k \sum_{r=1}^n \1_{\{d_r^{(n)}=k\}} - \frac 1{\E[D]} kp_k\Bigr|
$$
First note that by the strong law of large numbers
$$
\mathfrak m_1^{(n)}=\frac 1n \sum_{k=1}^n d_k^{(n)}\to \E[D]>0, \text{ almost surely}.
$$
Moreover, 
$$
\E\Bigl[\frac1n \sum_{k=1}^n \1_{\{d_k^{(n)}\ge a_n\}}  d_k^{(n)} \Bigr]\le \frac 1n + \E[\1_{\{q \ge a_n-1\}}q] = \cO(n^{-\frac{\tau-2}{\tau}})
$$
so that  $\mathfrak p_n=\cO_P(n^{-\frac{\tau-2}{\tau}})$.
Next, analyse the third term:
$$
\mathfrak m_2^{(n)}= \frac 1{n \,\mathfrak m_1^{(n)}}\sum_{k=1}^n \1_{\{d_k^{(n)}<a_n\}} (d_k^{(n)})^2\le  \frac 1{n \,\mathfrak m_1^{(n)}}\Bigl(\sum_{k=1}^n \1_{\{q_k<a_n\}} q_k^2 +2q_n+1\Bigr).
$$
Standard arguments yield that
$$\frac 1n \E\Bigl[\sum_{k=1}^n \1_{\{q_k<a_n\}} q_k^2 \Bigr] =\E\bigl[ \1_{\{q<a_n\}} q^2\bigr]=\begin{cases} \cO(n^{\frac {3-\tau}{\tau}}), &\text{ if }\tau <3,\\
\cO(\log n), & \text{ if }\tau=3,\\
\cO(1), & \text{ if }\tau>3,
\end{cases}
$$
so that 
$$
\mathfrak m_2^{(n)}=\begin{cases} \cO_P(n^{\frac {3-\tau}{\tau}}), &\text{ if }\tau <3,\\
\cO_P(\log n), & \text{ if }\tau=3,\\
\cO_P(1), & \text{ if }\tau>3.
\end{cases}
$$
In the case that $\sum_{r=1}^n q_r\ge1$ we further have that
\begin{align*}
\mathfrak d^{(n)}=&\ \frac12 \sum_{k=1}^\infty \Bigl| \frac 1{n \mathfrak m_1^{(n)}} k \sum_{r=1}^n \1_{\{d_r^{(n)}=k\}} - \frac 1{\E[D]} kp_k\Bigr|\\
\le &\ \frac12 \frac 1{\E[D]} \sum_{k=1}^\infty  k\Bigl| \frac 1n \sum_{r=1}^n \1_{\{d_r^{(n)}=k\}} -  p_k\Bigr|+\frac12 \Bigl|1-\frac {\mathfrak m_1^{(n)}}{\E[D]}\Bigr| \\
\le&\  \frac12 \frac 1{\E[D]} \sum_{k=1}^\infty  k\Bigl| \frac 1n \sum_{r=1}^n \1_{\{q_r=k\}} -  p_k\Bigr|+\frac 1{\E[D] \,n} (q_n+1)+\frac12 \Bigl|1-\frac {\mathfrak m_1^{(n)}}{\E[D]}\Bigr| 
\end{align*}
Letting $\bar S^{(n)}_k= \frac 1n \sum_{r=1}^n \1_{\{q_r=k\}}$ we get that 
$$
\E[|\bar S^{(n)}_k-p_k|]\le \var(\bar S^{(n)})^{1/2}\wedge (2p_k) \le \sqrt{\frac{p_k}{n}}\wedge (2p_k).
$$
Hence,
$$
\E\Bigl[\sum_{k=1}^\infty k |\bar S_k^{(n)}-p_k|\Bigr]\le \sum_{k=1}^\infty k \sqrt{\frac{p_k}{n}}\wedge (2p_k)=\begin{cases} \cO(n^{-\frac {\tau-2}\tau}), & \text{ if } \tau<4,\\
\cO(n^{-1/2} \log n), &\text{ if } \tau=4,\\
\cO(n^{-1/2}), &\text{ if }\tau >4,
\end{cases}
$$
where the $\cO$ estimates follow with standard computations.
In the case that $\tau>3$ the random variables $q_1,\dots$ have a finite second moment and we get that
$$
|\mathfrak m_1^{(n)}-\E[D]|=\cO_P(n^{-1/2}).
$$
If $\tau<3$, we let $\bar {\mathfrak m}_1^{(n)}=\frac 1n \sum_{k=1}^n q_k \1_{\{q_k\le \lceil n^{1/(\tau-1)}\rceil\}} $
\begin{align*}
\E[|\mathfrak m_1^{(n)}- \E[D]|]&\le \E\bigl[\mathfrak m_1^{(n)}-\bar {\mathfrak m}_1^{(n)}\bigr] +\E\bigl[\bigl|\bar {\mathfrak m}_1^{(n)}- \E[\bar {\mathfrak m}_1^{(n)}]\bigr|\bigr]\\
&\qquad + \E[D]- \E\bigl[\bar {\mathfrak m}_1^{(n)}\bigr].
\end{align*}
We bound the term in the middle by
$$
\E\bigl[\bigl|\bar {\mathfrak m}_1^{(n)}- \E[\bar {\mathfrak m}_1^{(n)}]\bigr|\bigr] \le \var (\bar {\mathfrak m}_1^{(n)})^{1/2} \le \frac 1{\sqrt n} \Bigl(\sum_{k=1}^{\lceil n^{1/(\tau-1)}\rceil} k^2 p_k\Bigr)^{1/2}
$$
and get with standard computations that all three terms are of order $\cO(n^{-\frac{\tau-2}{\tau-1}})$ which is of lower order than $n^{-\frac{\tau-2}\tau}$.
Consequently, 
$$
\mathfrak d^{(n)}=\begin{cases} \cO_P(n^{-\frac {\tau-2}{\tau}}), & \text{ if } \tau<4,\\
\cO_P(n^{-1/2}\log n), & \text{ if } \tau=4,\\
\cO_P(n^{-1/2}), & \text{ if }\tau>4.\end{cases}
$$
We can couple $\mathbf d^{(1)},\mathbf d^{(2)},\dots$ in such a way that the above $\cO_P$ estimates hold actually almost surely as $\cO$-estimates  (in the sense that the constant is allowed to be random). For the appropriately coupled sequences one gets that, almost surely,
$$
\ell_n \mathfrak p_n\to 0, \ \ell_n \frac {\mathfrak m^{(n)}_2}{\sqrt n}\to 0 \text{ \ and \ } \ell_n \mathfrak d^{(n)}\to 0.
$$
Choosing $G_n^{o_n}$ as random graph with fixed degree sequence $\mathbf d^{(n)}$ we can apply Theorem~\ref{theo:12} $\omega$-wise and get the result.
%
%
%
\end{proof}

\end{appendix}

\bibliography{bibliografie}
\bibliographystyle{alpha}

\end{document}